\documentclass[11pt,oneside,english,jou]{amsart}
\usepackage[utf8]{inputenc} 
\usepackage[T1]{fontenc}
\usepackage{babel}
\usepackage{verbatim}
\usepackage{mathrsfs}
\usepackage{amstext}
\usepackage{amsthm}
\usepackage{amssymb}
\usepackage{amsfonts}
\usepackage{amsmath}
\usepackage{amscd}
\usepackage{esint}
\usepackage{enumerate}
\usepackage[unicode=true,pdfusetitle,
bookmarks=true,bookmarksnumbered=false,bookmarksopen=false,
breaklinks=false,pdfborder={0 0 1},colorlinks=false]
{hyperref}
\usepackage[margin=2.5cm]{geometry}
\usepackage[foot]{amsaddr}
\usepackage{mathtools}
\usepackage{dsfont}
\usepackage[shortlabels]{enumitem}
\usepackage{bm}
\usepackage{caption}

\usepackage[markup=underlined,commandnameprefix=always]{changes}

\definechangesauthor[color=blue]{YG}

\numberwithin{figure}{section}
\theoremstyle{plain}
\newtheorem{thm}{\protect\theoremname}[section]
\theoremstyle{definition}
\newtheorem{rem}[thm]{\protect\remarkname}
\theoremstyle{definition}
\newtheorem{defn}[thm]{\protect\definitionname}
\theoremstyle{plain}
\newtheorem{prop}[thm]{\protect\propositionname}
\theoremstyle{plain}
\newtheorem{lem}[thm]{\protect\lemmaname}
\theoremstyle{plain}

\theoremstyle{plain}
\newtheorem{cor}[thm]{\protect\corollaryname}

\theoremstyle{definition}

\theoremstyle{definition}

\theoremstyle{definition}

\theoremstyle{definition}

\theoremstyle{plain}
\newtheorem*{SSOYconjecture2}{Schroer--Sauer--Ott--Yorke prediction error conjecture}

\usepackage{tikz}
\usetikzlibrary{shapes,arrows}
\usepackage{verbatim}
\usepackage{amsthm}
\usepackage{amstext}

\DeclareMathOperator{\diam}{diam}
\DeclareMathOperator{\dist}{dist}

\DeclareMathOperator{\Leb}{Leb}

\DeclareMathOperator{\supp}{supp}

\newcommand{\D}{\mathbb D}

\newcommand{\R}{\mathbb R}
\newcommand{\Z}{\mathbb Z}
\newcommand{\N}{\mathbb N}

\renewcommand{\Im}{\textup{Im}}

\newcommand{\eps}{\varepsilon}

\newcommand{\mF}{\mathcal{F}}
\newcommand{\mH}{\mathcal{H}}

\newcommand{\mB}{\mathcal{B}}

\newcommand{\mE}{\mathcal{E}}

\newcommand{\mR}{\mathcal{R}}

\newcommand{\udim}{\overline{\dim}_B\,}
\newcommand{\ldim}{\underline{\dim}_B\,}
\newcommand{\bdim}{\dim_B}

\newcommand{\hdim}{\dim_H}
\newcommand{\lhdim}{\underline{\dim}_H}

\newcommand{\uid}{\overline{\mathrm{ID}}}
\newcommand{\lid}{\underline{\mathrm{ID}}}
\newcommand{\idim}{\mathrm{ID}}
\newcommand{\lbdim}{\underline{\dim}_B\,}

\newcommand{\cdim}{\dim_{\mathrm{c}}}
\newcommand{\ld}{\underline{d}}
\newcommand{\ud}{\overline{d}}

\DeclareMathOperator{\Ker}{Ker}

\DeclareMathOperator{\Lip}{Lip}
\DeclareMathOperator{\Per}{Per}
\DeclareMathOperator{\Prep}{PrePer}

\DeclareMathOperator{\rank}{rank}

\providecommand{\conjecturename}{Conjecture}
\providecommand{\corollaryname}{Corollary}
\providecommand{\definitionname}{Definition}
\providecommand{\examplename}{Example}
\providecommand{\lemmaname}{Lemma}
\providecommand{\problemname}{Problem}
\providecommand{\propositionname}{Proposition}
\providecommand{\remarkname}{Remark}
\providecommand{\theoremname}{Theorem}
\providecommand{\taskname}{Task}

\newcommand{\lam}{\lambda}

\newcommand{\om}{\omega}

\def\Om{\Omega}

\def\N{{\mathbb N}}

\def\Sk{{\mathcal S}}

\def\be{\begin{equation}}
	\def\ee{\end{equation}}

\newcommand{\Ek}{{\mathcal E}}

\def\ldim{\underline{\dim}}

\author[K. Bara\'{n}ski]{Krzysztof Bara\'{n}ski$^*$}
\address{$^*$Institute of Mathematics, University of Warsaw, ul.~Banacha 2, 02-097 Warszawa, Poland}
\email{baranski@mimuw.edu.pl}

\author[Y. Gutman]{Yonatan Gutman$^\dagger$}
\address{$^\dagger$Institute of Mathematics, Polish Academy of Sciences,
	ul.~\'Sniadeckich 8, 00-656 Warszawa, Poland}
\email{gutman@impan.pl}

\author[A. \'{S}piewak]{Adam \'{S}piewak$^\dagger$}

\email{ad.spiewak@gmail.com}

\title{Predicting dynamical systems with too few time-delay measurements: error estimates}

\thanks{YG and A\'S were partially supported by the National Science Centre (Poland) grant 2020/39/B/ST1/02329.}

\begin{document}

\maketitle

\begin{abstract}
We study the problem of reconstructing and predicting the future of a dynamical system by the use of time-delay measurements of typical observables. 
Considering the case of too few measurements, we prove that for Lipschitz systems on compact sets in Euclidean spaces, equipped with an invariant Borel probability measure $\mu$ of Hausdorff dimension $d$, one needs at least $d$ measurements of a typical (prevalent) Lipschitz observable for $\mu$-almost sure reconstruction and prediction. Consequently, the Hausdorff dimension of $\mu$ is the precise threshold for the minimal delay (embedding) dimension for such systems in a probabilistic setting. Furthermore, we establish a lower bound postulated in the Schroer--Sauer--Ott--Yorke prediction error conjecture from 1998, after necessary modifications (whereas the upper estimates were obtained in our previous work). To this aim, we prove a general theorem on the dimensions of conditional measures of $\mu$  with respect to time-delay coordinate maps.
\end{abstract}

\section{Introduction}

\subsection{Time-delayed measurements}

The paper considers the problem of reconstructing or predicting an (unknown) future of a dynamical system by time-delayed measurements of observables. This is one of the central themes in non-linear data analysis, leading to non-trivial mathematical questions related to practical tools and algorithms used in applications. Consider a \emph{phase space} $X$ (the set of all possible states of the system) and a deterministic dynamics on $X$, generated by 
a transformation $T \colon X \to X$, which defines a one-step \emph{evolution rule}. Let $h \colon X \to \R$ be an \emph{observable}, which can be seen as a function measuring certain parameter of the system. We assume that the observer has no direct access to the original system $(X,T)$ and its (finite) orbits
\[ x, Tx, T^2 x,\ldots, T^m x, \qquad x \in X,\;  m \in \N,\]
whereas the knowledge of the system is derived from observations (measurements) of the values of $h$ along the orbits, that is
\begin{equation}\label{eq: time series} h(x), h(Tx), h(T^2 x), \ldots, h(T^m x).
\end{equation}
One of the main tasks arising in this context is \emph{reconstructing} the unknown dynamics $(X,T)$ based on the observational data \eqref{eq: time series}. In particular, one would like to \emph{predict} the future values $h(T^{m+1} x), h(T^{m+2} x, \ldots)$ from the \emph{time series} given by \eqref{eq: time series}. One of the effective approaches to these problems is to construct a \emph{time-delayed model} of the system in a high-dimensional Euclidean space by the use of the time series \eqref{eq: time series}, hoping that this will unfold the original dynamics. More precisely, fixing a positive integer $k$ (called \emph{delay length} or \emph{delay dimension}) for the dimension of the \emph{reconstruction space} $\R^k$, one transforms \eqref{eq: time series} into a sequence of points $y_j \in \R^k$ as
\begin{equation}\label{eq:y_j}
y_j = (h(T^j x), h(T^{j+1} x), \ldots, h(T^{j+k-1} x)) \qquad \text{ for } j=0, \ldots, m-k+1. 
\end{equation}
One expects that if $k$ is large enough, then the original dynamics $(T,X)$ can be reliably modelled or approximated by the \emph{observed dynamics} $y_j \mapsto y_{j+1}$ in $\R^k$. This idea has proved fruitful in applications, see e.g.~\cite{sm90nonlinear,SeaClutter,hgls05distinguishing,BioclimaticBuildings,  BaghReddy23}, and a mathematical theory has been built to understand the underlying mechanisms and conditions that determine its performance, see e.g.~\cite{PCFS80,FarmerSidorowich87,KY90,SYC91,KBA92,SSOY98,Voss03,Rob11,HBS15} and the references therein. 

The main object of interest in this theory is the $k$-\emph{delay coordinate map} corresponding to the observable $h$, defined as
\begin{equation}\label{eq:coord_map}
\phi_{h,k} \colon X \to \R^k, \qquad \phi_{h,k} (x) = (h(x), h(Tx), \ldots, h(T^{k-1}x)).
\end{equation}
We will frequently suppress the dependence of the delay coordinate map on $k$, writing $\phi_h$ instead of $\phi_{h,k}$. Note that in this notation, the reconstruction space time series \eqref{eq:y_j} is given as $y_j = \phi_h(T^j x)$ and the problem of predicting future values of the time series \eqref{eq: time series} for an observable $h$ translates to the question whether $\phi_h(x)$ determines $\phi_h(Tx)$ uniquely. Similarly, the problem of reconstructing the dynamics $(X,T)$ by time-delay measurements of $h$ translates to the injectivity properties of $\phi_h$. In terms of the dynamical systems theory, these two conditions on $\phi_h$ can be rephrased as being, respectively, a factor map (semi-conjugation) and an isomorphism (conjugation) between the original system and its image under $\phi_h$, in a category (regularity) depending on the context. In both cases, the diagram
\begin{equation}\label{eq:comm_diagram}
\begin{CD}
	X @>T>> X\\
	@VV\phi_h V @VV \phi_h V\\
	\phi_h(X) @>S_h>> \phi_h(X)
\end{CD}
\end{equation}
commutes, with the well-defined map $S_h$ given as
\begin{equation*}\label{eq:prediction map}
S_h(h(x), h(Tx), \ldots, h(T^{k-1}x)) = (h(Tx), h(T^2 x), \ldots, h(T^k x)). 
\end{equation*}
On the reconstruction space time series \eqref{eq:y_j} the map $S_h$ acts as $S_h(y_j) = y_{j+1}$, realizing a one-step prediction of the next term in \eqref{eq: time series}. In view of this, we formulate the following definition.

\begin{defn}\label{defn:predict-embed}
An observable $h\colon X \to\R$ is (\emph{deterministically}) $k$-\emph{predictable}, if for the $k$-delay coordinate map $\phi_h\colon X \to \R^k$ given by \eqref{eq:coord_map} there exists a map $S_h\colon \phi_h(X) \to \phi_h(X)$, called the \emph{prediction map}, such that the diagram \eqref{eq:comm_diagram} commutes.
\end{defn}
Note that if $\phi_h$ is injective, then the diagram \eqref{eq:comm_diagram} exists, with $S_h$ defined as $S_h = \phi_h\circ T \circ \phi_h^{-1}$.  In this case, the conjugated system $(\phi_h(X), S_h)$ can be treated as a faithful \emph{reconstruction model} of the original dynamical system $(X,T)$ in the reconstruction space $\R^k$, based only on the time-delay values of the observable $h$. Furthermore, if the $k$-delay coordinate map $\phi_h$ is continuous and injective on a compact set $X$, then it is a topological embedding of $X$ onto $\phi_h(X) \subset \R^k$. In this case we say that $\phi_h$ is a $k$-\emph{delay embedding}.

There is a long history of mathematical results on the embedding properties of $\phi_h$ for typical observables $h$. These are usually referred to as \emph{Takens-type embedding theorems}, as they are related to a classical result of Floris Takens \cite{T81} from 1981. In an extended version by Huke \cite{Huke-report,StarkEmbedSurvey} (see also~\cite{N91}), it states that if $T\colon X \to X$ is a $C^1$-diffeomorphism of a compact $C^1$-manifold $X$ such that all periodic orbits of $T$ of period smaller than $2\dim X$ are isolated and hyperbolic and each has distinct eigenvalues, then for a generic $C^1$-observable $h \colon X \to \R$, the $k$-delay coordinate map $\phi_h$ is a  $C^1$-embedding for $k>2\dim X$. This seminal result was extended in a number of papers, including \cite{SYC91, StarkEmbedSurvey, 1999delay, CaballeroEmbed, TakensPlato03, StarkStochEmbed, Rob05, Gut16, GQS18, BGS20, NV18}, to various settings and categories of systems. Apart from that, the predictability problem was considered, among others, in \cite{T02,SSOY98, BGS22,BGSPredict,KoltaiKunde}. The results obtained in the references mentioned above are considered to be a theoretical basis for procedures of time-delay reconstruction and prediction of the dynamics, used in applications, see e.g.~\cite{Rainfall94, hgls05distinguishing, Streamflow07, SolarCycle18, Dlotko_et_al}. A common feature in these papers is that the bound on the delay dimension $k$ is related to the dimension of the phase space $X$. Let us state two results which are relevant to the present work. Following \cite{SYC91}, they are presented in the category of Lipschitz transformations and observables on compact sets in Euclidean spaces, where genericity is understood as the \emph{prevalence} with a polynomial probe set (see Subsection~\ref{subsec:preval}). Below, the symbol $\udim$ denotes the upper box dimension - see Subsection~\ref{subsec:dim} for the definitions of all notions of dimensions that appear in this section and the relations between them.
 
\begin{thm}[{\bf Time-delay prediction and embedding theorem}{}]\label{thm:predict_determ}
Let $X \subset \R^N$, $N \in \N$, be a compact set and let $T\colon X \to X$ be a Lipschitz map. Then a prevalent Lipschitz observable $h \colon X \to \R$ is $k$-predictable, with a continuous prediction map, for every $k > 2\udim X$. If, additionally, $T$ is injective and $2\udim (\{ x \in X : T^p x = x \} )< p$ for $p=1, \ldots, k-1$, then the $k$-delay coordinate map $\phi_h$ is a $k$-delay embedding for a prevalent Lipschitz observable $h \colon X \to \R$.
\end{thm}

The first part of the theorem is \cite[Theorem~1.16]{BGSPredict}, while the second one comes from \cite{SYC91} (see \cite[Theorem~4.5]{Rob11} for the above formulation). In fact, the results hold with greater generality than above, see \cite{SYC91, Rob11, BGSPredict} for details. Theorem~\ref{thm:predict_determ} and other mentioned results show that a threshold for the delay dimension that is sufficient for a reliable prediction and reconstruction of the system is roughly equal to twice the dimension of the phase space, which agrees with the well-known Menger--N\"obeling and Whitney theorems on embedding, respectively, topological spaces and smooth manifolds into Euclidean spaces.

\subsection{Time-delay measurements from a probabilistic point of view} 

It turns out that the minimal number of measurements required for a reliable prediction or reconstruction of the system can be reduced (at least) by half in a probabilistic setting, when an observer is interested only in the `almost sure' behaviour of the system. Mathematically speaking, this means that one assumes $X$ to be endowed with a probability measure, studying trajectories of almost all points $x \in X$. The relevance of this approach was conjectured by Schroer, Sauer, Ott and Yorke in \cite{SSOY98}. In a series of our previous papers \cite{BGS20, BGS22, BGSPredict} we developed a theory of almost sure time-delay prediction and embedding for (locally) Lipschitz systems with a Borel probability measure on compact or Borel sets in Euclidean spaces. Within this approach, the analogues of the notions of the time-delay predictability and time-delay embedding are the following.

\begin{defn}\label{eq:almost-sure-embedding} Let $X \subset \R^N$, $N \in \N$, be a Borel set, let $\mu$ be a Borel probability measure on $X$ and let $T\colon X \to X$ be a Borel transformation. We say that a Borel observable $h\colon X \to \R$ is \emph{almost surely $k$-predictable}, if for the $k$-delay coordinate map $\phi_h$, there exists a Borel set $X_h \subset X$ of full $\mu$-measure and a map $S_h\colon\phi_h(X_h) \to \phi_h(X)$ (\emph{prediction map}) such that the diagram
\begin{equation}\label{eq:comm_diagram2}
\begin{CD}
X_h @>T>> X\\
@VV\phi_h V @VV \phi_h V\\
\phi_h(X_h) @>S_h>> \phi_h(X)
\end{CD}
\end{equation}
commutes. 
\end{defn}

\begin{rem} If, additionally, $T$ is continuous and $h$ is an almost surely $k$-predictable continuous observable, then the set $X_h$ can be chosen such that the set $\phi_h(X_h)$ and the prediction map $S_h$ are Borel (see \cite[Proposition~1.13]{BGSPredict}).
\end{rem}

\begin{defn}
We say that the $k$-delay coordinate map $\phi_h$ is \emph{almost surely injective}, if it is injective on a full $\mu$-measure Borel set $X_h \subset X$. Then the commuting diagram \eqref{eq:comm_diagram2} exists, with $S_h = \phi_h\circ T \circ (\phi_h|_{X_h})^{-1}$.

If $h$ is almost surely $k$-predictable on a full-measure Borel set $X_h \subset X$, which is $T$-invariant, i.e.~$T(X_h) \subset X_h$, then the diagram \eqref{eq:comm_diagram2} has the form
\begin{equation}\label{eq:comm_diagram3}
\begin{CD}
X_h @>T>> X_h\\
@VV\phi_h V @VV \phi_h V\\
\phi_h(X_h) @>S_h>> \phi_h(X_h)
\end{CD},
\end{equation} 
which provides a semi-conjugation between the system $(X,T)$ restricted to a full-measure subset of the phase space and its model in $\R^k$. Similarly, if $\phi_h$ 
is almost surely injective on a $T$-invariant full-measure Borel set $X_h \subset X$, then the diagram \eqref{eq:comm_diagram3} provides a measurable isomorphism between the system $(X_h, \mu, T|_{X_h})$ and its \emph{measurable model} $(\phi_h(X_h),\phi_h\mu, S_h)$ in $\R^k$, where $\phi_h\mu$ denotes the push-forward of $\mu$ under the map $\phi_h$ (see Subsection~\ref{subsec:meas}). In the case when $\mu$ is additionally $T$-invariant, the map $\phi_h$ becomes an isomorphism in the category of measure-preserving transformations, hence we call it a \emph{$k$-delay measure-preserving-transformations isomorphism} (in short, \emph{$k$-delay mpt isomorphism}).
\end{defn}

In \cite[Theorem~1.2 and Remark~1.3]{BGS20} and \cite[Theorem~1.18]{BGSPredict} we proved the following.

\begin{thm}[{\bf Probabilistic time-delay prediction and embedding theorem}{}]\label{thm:predict_prob}
Let $X \subset \R^N$, $N \in \N$, be a compact set, let $\mu$ be a Borel probability measure on $X$ and let $T\colon X \to X$ be a Lipschitz map. 
Then the following hold.
\begin{enumerate}[$($a$)$] 
\item A prevalent Lipschitz observable $h \colon X \to \R$ is almost surely $k$-predictable for every $k > \hdim \mu$. Furthermore, if $k > \hdim(\supp \mu)$, then the set $X_h$ can be chosen such that the prediction map $S_h$ is continuous. If, additionally, $\mu$ is $T$-invariant, then $X_h$ can be chosen  to be $T$-invariant.
\item 
If $T$ is injective and $\hdim ( \mu|_{\{ x \in X : T^p x = x \}})< p$ for $p=1, \ldots, k-1$, where $k > \hdim \mu$, then the $k$-delay coordinate map $\phi_h$ is almost surely injective for a prevalent Lipschitz observable $h \colon X \to \R$. 
If, additionally, $\mu$ is $T$-invariant, then $X_h$ can be chosen to be $T$-invariant and, consequently, $\phi_h$ is a $k$-delay mpt isomorphism. Furthermore, if $\mu$ is $T$-invariant and ergodic, then the assumption $\hdim ( \mu|_{\{ x \in X : T^p x = x \}})< p$ for $p=1, \ldots, k-1$ may be omitted.
\end{enumerate}
\end{thm}
Here and in the sequel the symbol $\supp$ denotes the topological support of the measure (see Subsection~\ref{subsec:meas}). Theorem~\ref{thm:predict_prob} shows that the threshold for the delay dimension that is sufficient for an almost sure prediction and reconstruction of the system is roughly equal to the dimension of the phase space or a given measure, which reduces (at least) by half the number of required measurements, compared to the deterministic setup.

\subsection{New results}
The main objective of this paper is to study the case of (too) few measurements ("undersampling"), mainly in the probabilistic setting, which roughly means $k < \dim_H \mu$. For such values of $k$ one should not expect a faithful $k$-delay prediction or reconstruction of the system. Our results confirm this intuition in a rigorous way, yielding a natural theoretical limitation for these time-delay measurement procedures. To our knowledge, these are the first known results showing lower bounds for the minimal delay dimension needed for a faithful prediction or reconstruction of the system. We should emphasize that the injectivity of a single  delay coordinates map or predictability of a single observable may hold even in the case of few measurements. Our results exclude such possibility for a typical (prevalent) observable. 

The first result concerns the setting of Theorem~\ref{thm:predict_prob} and shows that for Lipschitz systems on compact spaces with an invariant probability measure $\mu$, the Hausdorff dimension of $\mu$ is the precise threshold for the minimal delay dimension needed for system reconstruction and prediction procedures for typical Lipschitz observables in the probabilistic setting. To state the result, let us denote the sets of periodic and pre-periodic points as
\begin{align*}
\Per_p(T) &= \{ x \in X : T^p x = x \},\\
\Prep_p(T) &= \{ x \in X : T^p x \in \{ x, Tx, \ldots, T^{p-1} x \}\}
\end{align*}
for $p \in \N$.

\begin{thm}\label{thm:excluding embed and predict}
Let $X \subset \R^N$, $N \in \N$, be a compact set, let $\mu$ be a Borel probability measure on $X$ and let $T\colon X \to X$ be a Lipschitz map. Then the following hold.
\begin{enumerate}[$($a$)$]
\item\label{it:prob non injective} If $k < \hdim T^{k-1}\mu$, then for a prevalent Lipschitz observable $h \colon X \to \R$, the $k$-delay coordinate map $\phi_h$ is not almost surely injective. If, additionally, $\mu$ is $T$-invariant or $T$ is bi-Lipschitz onto its image, then the assumption $k < \hdim T^{k-1}\mu$ may be replaced by $k < \hdim \mu$.
\item\label{it:prob non predict} If $k < \hdim T^k (\mu |_{ X \setminus \Prep_k (T)})$, then a prevalent Lipschitz observable $h \colon X \to \R$ is not almost surely $k$-predictable. If, additionally, $T$ is bi-Lipschitz onto its image or $\mu$ is $T$-invariant, then the assumption $k < \hdim T^k (\mu |_{ X \setminus \Prep_k (T)})$ may be replaced by $k < \hdim \mu |_{ X \setminus \bigcup_{p=1}^k \Per_p (T)}$. Furthermore, if  $\mu$ is $T$-invariant and ergodic, then it suffices to assume $k < \hdim \mu$.
\end{enumerate}
\end{thm}

\begin{rem}
Note that in the general case of a non-invariant measure $\mu$, in Theorem~\ref{thm:excluding embed and predict} it is not enough to assume $k < \hdim \mu$ in order to exclude the almost sure injectivity of $\phi_h$ and almost sure $k$-predictability of $h$, see examples in Section~\ref{sec:iterate examples}. 
\end{rem}

\begin{rem}\label{rem:assum preper} It is easy to see that some restrictions on the size of the set of preperiodic points $\Prep_k (T)$ in Theorem~\ref{thm:excluding embed and predict} are necessary. A trivial example is when $T$ is the identity -- in this case every observable is predictable. More generally, if there exist $p,k \in \N$, $p \le k-1$, such that $T^px = T^kx$ for every (resp.~$\mu$-almost every) $x \in X$, then every observable is $k$-predictable (resp.~$\mu$-almost predictable). This holds, e.g.,~for rational rotations on a $d$-dimensional torus or for $X = [0,1]$, $Tx = |x - 1/2|$.  
\end{rem}

Combining Theorems \ref{thm:predict_prob} and \ref{thm:excluding embed and predict} we see that $\hdim \mu$ is the precise threshold for the minimal delay (embedding) dimension required for both almost sure reconstruction and prediction of a system. In practical tasks the dimension of the system is often unknown and one is faced with a challenge of estimating the embedding dimension from a time series. Several algorithms aiming at this and similar tasks were introduced, e.g. the false nearest neighbour algorithm and its modifications \cite{KBA92, CenysPyragas88, LiebertPawelzikSchuster91, Dlotko_et_al} (see also \cite{Abarbanel_book} and \cite[Section 1.5]{BGSPredict} for a more detailed discussion). It is an interesting problem to study rigorously the behaviour of such algorithms and its relation to the presented results.

Theorem~\ref{thm:excluding embed and predict} implies the following result concerning predictability and embedding limitation in the deterministic setting of Theorem~\ref{thm:predict_determ}. Note that unlike in the probabilistic case, we do not obtain the same threshold for the minimal delay dimension as the one which appears in Theorem~\ref{thm:predict_determ} (i.e.~$2\udim X$). 

\begin{thm}\label{thm:excluding embed and predict determ}
Let $X \subset \R^N$, $N \in \N$, be a compact set and let $T\colon X \to X$ be a Lipschitz map. Then the following hold.
\begin{enumerate}[$($a$)$]
\item \label{it:determ non injective} If $k < \hdim T^{k-1}(X)$, then for a prevalent Lipschitz observable $h \colon X \to \R$, the $k$-delay coordinate map $\phi_h$ is not injective.
\item \label{it:determ non predict} If $k < \hdim\big( T^k(X) \setminus \bigcup_{p=1}^k\Per_p(T) \big)$, then a prevalent Lipschitz observable $h \colon X \to \R$ is not $k$-predictable.
\end{enumerate}
\end{thm}

The proofs of Theorems~\ref{thm:excluding embed and predict} and~\ref{thm:excluding embed and predict determ} are presented in Section~\ref{sec:cond measures imply}.

The subsequent results of the paper are related to a conjecture of Schroer, Sauer, Ott and Yorke from \cite{SSOY98} concerning error bounds in time-delay prediction procedures. To formulate the conjecture, let $X \subset \R^N$, $N \in \N$ be a Borel set, let $\mu$ be a Borel probability measure on $X$ and let $T \colon X \to X$ be a Borel transformation. For a Borel observable $h \colon X \to \R$, $y \in \supp \phi_h\mu$ and $\eps>0$, define
	\begin{align*}
		\chi_{h,\eps}(y) &=  \frac{1}{\mu(\phi_h^{-1}(B(y, \eps)))} \int_{\phi_h^{-1}(B(y, \eps))} \phi_h\circ T \: d\mu,\\
		\sigma_{h,\eps}(y) &= 
		\bigg(\frac{1}{\mu(\phi_h^{-1}(B(y, \eps)))} \int_{\phi_h^{-1}(B(y, \eps))} \|\phi_h \circ T -\chi_{h,\eps}(y)\|^2d\mu\bigg)^{\frac{1}{2}}
	\end{align*}
(provided the integrals exist). 

\begin{rem}\label{rem:two kinds of predict}
In \cite{SSOY98}, the probabilistic notion of predictability was introduced in another way, by defining $k$-\emph{predictable points} as points $y \in \R^k$ for which the \emph{prediction error}
\[
\sigma_h(y) = \lim_{\eps \to 0}\sigma_{h,\eps}(y)
\]
exists and is equal to $0$. In \cite[Theorem~1.14]{BGSPredict} we showed that for continuous systems on compact spaces, a continuous observable $h$ is almost surely $k$-predictable with respect to $\mu$ if and only if $\phi_h\mu$-almost every point $y \in \R^k$ is $k$-predictable (see also Lemma \ref{lem: cond meas iff} below). Hence, in the setup considered in this paper, the two notions of almost sure predictability coincide.
\end{rem}

\begin{rem} The almost sure predictability is closely related to the behaviour of the Farmer--Sidorowich prediction algorithm \cite{FarmerSidorowich87} (see \cite[Proposition~1.3]{BGSPredict}). More precisely, whenever almost sure $k$-predictability holds with respect to an ergodic measure $\mu$, it implies an almost sure convergence of the suitable Farmer--Sidorowich algorithm. Refer to \cite[Corollary~1.15]{BGSPredict} for details. 
\end{rem}

In \cite{SSOY98}, Schroer, Sauer, Ott and Yorke conjectured \cite[Conjecture 2]{SSOY98} a decay rate of the prediction errors $\sigma_{h,\eps}(y)$ for typical observables.  The conjecture was stated for a special class of invariant measures on attractors for smooth diffeomorphisms on Riemannian manifolds -- so called \emph{natural measures}, see Definition \ref{defn:nat_measure}.

\begin{SSOYconjecture2}\label{con:2}
	Let $T\colon M \to M$ be a smooth diffeomorphism of a compact Riemannian manifold $M$ with a compact $T$-invariant attractor $X \subset M$ and a natural measure $\mu$ on $X$ of information dimension $\idim(\mu) = D$. Fix a generic observable $h\colon M \to \R$, $k \in \N$ and a small $\delta > 0$. Then for the $k$-delay coordinate map $\phi_h$ corresponding to  $h$ and sufficiently small $\eps>0$, the following hold.
	\begin{enumerate}[$($i$)$]
		\item\label{it: ssoy2 i} If $k < D$, then 
		\[
		\mu(\{x \in X: \sigma_{h,\eps}(\phi_h(x)) > \delta\}) \geq C \quad \text{ for some } C > 0.
		\]
		
		\item\label{it: ssoy2 ii} If $D < k < 2D$ and $\phi_h$ is not injective, then 
		\[
		C_1  \eps^{k-D} \le \mu(\{x \in X: \sigma_{h,\eps}(\phi_h(x)) > \delta\}) \le C_2 \eps^{k-D} \quad \text{ for some } C_1, C_2 > 0.
		\]
		\item\label{it: ssoy2 ii injective} If $D < k < 2D$ and $\phi_h$ is injective, then 
		\[
		\mu(\{x \in X: \sigma_{h,\eps}(\phi_h(x)) > \delta\}) = 0.
		\]
		\item\label{it: ssoy2 iii} If $k > 2D$, then 
		\[
		\mu(\{x \in X: \sigma_{h,\eps}(\phi_h(x)) > \delta\}) = 0.
		\]
	\end{enumerate}
\end{SSOYconjecture2}

Note that the first conjecture by Schroer, Sauer, Ott and Yorke (\cite[Conjecture~1]{SSOY98}, concerning the possibility of generic almost sure prediction in the case $k > D$, was settled in our previous papers \cite[Corollaries~1.9--1.10, Theorem~1.11]{BGS22} and \cite[Theorem~1.4]{BGSPredict} (see also Theorem~\ref{thm:predict_prob}).

\begin{rem}\label{rem: SSOY doubts}
In the original formulation of the prediction error conjecture in \cite{SSOY98}, as well as in the statement presented above, not all the details are precisely specified, which introduces the need for some interpretations. This includes the issue of the regularity class of the considered observables as well as the suitable notion of genericity. Following the approach used in \cite{SYC91}, we consider observables within the class of Lipschitz maps and prevalence as the notion of genericity, but  it should be noted that the obtained results apply also to the classes $C^r$, $r =1,\ldots, \infty$ of observables (see Definition~\ref{defn:preval} and the discussion afterwards). Another issue is the interdependence between $h$ and $\delta$ in the formulation of the conjecture, which has a non-trivial impact on the existence of the postulated lower bounds, see Remark~\ref{rem:order of quant} and Theorem~\ref{thm:counterexample}.
\end{rem}

In our previous paper \cite{BGSPredict}, we studied the upper bounds for the prediction error probability in the Schroer--Sauer--Ott--Yorke prediction error conjecture. We proved that the upper bound in assertion~\ref{it: ssoy2 ii} holds true (in a slightly weaker form) if we replace the information dimension with the upper box counting dimension of $\supp \mu$, i.e.~for $D = \udim (\supp\mu)$. Furthermore, assertion~\ref{it: ssoy2 iii} holds for $D = \ldim_B (\supp\mu)$ (hence also for $D = \udim (\supp\mu)$), while \ref{it: ssoy2 ii injective} is true whenever $\phi_h$ is injective (regardless of the dimension of the phase space). In fact, the results extend to a much broader class of Lipschitz systems on compact spaces $X \subset\R^N$ and arbitrary Borel probability measures $\mu$, see \cite[Theorem~1.6]{BGSPredict} for details. We  also provided examples \cite[Proposition~8.3]{BGSPredict} showing that the box counting dimensions cannot be replaced by $\idim(\mu)$ or $\hdim \mu$ in the general class of Lipschitz systems and Borel measures (in the case of natural measures on attractors for smooth diffeomorphisms the question remains open).

The second objective of this paper is to establish lower bounds in the Schroer--Sauer--Ott--Yorke prediction error conjecture, with a focus on assertion~\ref{it: ssoy2 i}, which corresponds to the case of too few measurements. Our main result (which strengthens Theorem~\ref{thm:excluding embed and predict}\ref{it:prob non predict}), shows that the assertion holds true for small enough $\delta > 0$, if the information dimension of the measure is replaced by its Hausdorff dimension, i.e.~after setting $D = \hdim \mu$. 

\begin{thm}\label{thm:ssoy2.i hdim}
Let $X \subset\R^N$, $N \in \N$, be a compact set, let $\mu$ be a Borel probability measure on $X$ and let $T\colon X \to X$ be a Lipschitz map. Fix $k \in \N$ such that $k < \hdim T^k ( \mu |_{ X \setminus \Prep_k (T)})$. Then for a prevalent Lipschitz observable $h \colon X \to \R$ one can find $C, \delta_0 > 0$ such that for every $0 < \delta < \delta_0$ there exists $\eps_0 > 0$ with
	\[ 	\mu(\{x \in X: \sigma_{h,\eps}(\phi_h(x)) > \delta\}) \geq C  \]
for every $0 < \eps < \eps_0$, where $\phi_h$ is the $k$-delay coordinate map corresponding to $h$. If, additionally, $T$ is bi-Lipschitz onto its image or $\mu$ is $T$-invariant, then the assumption $k < \hdim T^k (\mu |_{ X \setminus \Prep_k (T)})$ may be replaced by $k < \hdim \mu |_{ X \setminus \bigcup_{p=1}^k \Per_p (T)}$. Furthermore, if $T$-invariant and ergodic, then it suffices to assume $k < \hdim \mu$.
\end{thm}

\begin{rem}\label{rem:<>dim} If $h$ is almost surely $k$-predictable, then $\sigma_{h,\eps}(\phi_h(x)) \to 0$ as $\eps \to 0$ for $\mu$-almost every $x \in X$ (see Remark~\ref{rem:two kinds of predict}), so $\mu(\{x \in X: \sigma_{h,\eps}(\phi_h(x)) > \delta\}) \to 0$ as $\eps \to 0$ for every $\delta > 0$. Therefore (as the intersection of two prevalent sets is prevalent, in particular non-empty), Theorem~\ref{thm:predict_prob} shows that the assertion of Theorem~\ref{thm:ssoy2.i hdim} cannot hold for $k > \dim_H \mu$.
\end{rem}

The proof of Theorem~\ref{thm:ssoy2.i hdim} is presented in Section~\ref{sec:cond measures imply}. In Section~\ref{sec:examples} we present examples showing that the assumptions on the dynamics $T$ and the measure $\mu$ are necessary. We also provide an example showing that the result does not hold for information dimension (see Section~\ref{sec:no id}). However, we emphasize that this example lies within the general class of Lipschitz systems, and the considered measure is not a natural measure on an attractor for a smooth diffeomorphism. Hence, it does not provide a counterexample to assertion~\ref{it: ssoy2 i} of the Schroer--Sauer--Ott--Yorke prediction error conjecture.

Combining Theorem~\ref{thm:ssoy2.i hdim} with \cite[Theorem~1.6]{BGSPredict}, we obtain the following version of the Schroer--Sauer--Ott--Yorke prediction error conjecture, which holds true (after necessary modifications) for arbitrary Lipschitz systems on compact sets in Euclidean spaces, equipped with a Borel probability measure. Below we present the result in a slightly simplified version for invariant measures.

\begin{thm}[{\bf Prediction error estimates}{}]\label{thm:SSOY2}
Let $X \subset \R^N$, $N \in \N$, be a compact set, let $\mu$ a $T$-invariant Borel probability measure on $X$ and let $T\colon X \to X$ be a Lipschitz map. Assume $\mu(\bigcup_{p=1}^k\Per_p(T)) = 0$ and set $D_H = \hdim \mu$, $\overline{D}_B = \udim (\supp\mu)$, $\underline{D}_B = \lbdim (\supp\mu)$. Then for a prevalent Lipschitz observable $h \colon X \to \R$ and $k \in \N$, $\theta > 0$, one can find $C, \delta_0 > 0$ such that for every $0 < \delta < \delta_0$ there exists $\eps_0 > 0$ such that for the $k$-delay coordinate map $\phi_h$ and every $0 < \eps < \eps_0$, the following hold.
	\begin{enumerate}[$($i$)$]
		\item\label{it SSOY2 lower error bound} If $k < D_H$, then
		\[ 	\mu(\{x \in X: \sigma_{h,\eps}(\phi_h(x)) > \delta\}) \geq C.  \]
		\item\label{it:SSOY2 error bound} If $k >\overline{D}_B$ and $\phi_h$ is not injective, then
		\[ 
		\mu\left( \{ x \in X : \sigma_{h,\eps}(\phi(x)) > \delta \} \right) \leq C \eps^{k - \overline{D}_B - \theta}.
		\]
		\item\label{it:SSOY2 deterministic} If $k > 2\underline{D}_B$ or $\phi_h$ is injective, then 
		\[ 
		\{ x \in X : \sigma_{h,\eps}(\phi_h(x)) > \delta \} = \emptyset.
		\]
	\end{enumerate}
	If $\mu$ is additionally ergodic, then the assumption $\mu(\bigcup_{p=1}^k\Per_p(T)) = 0$ may be omitted.
\end{thm}

\begin{rem} Note that Theorem~\ref{thm:SSOY2} does not hold if we replace the dimensions $D_H$, $\overline{D}_B$, $\underline{D}_B$ in assertions (i)--(iii) by one notion of dimension, for any choice among $\idim(\mu)$, $D_H$, $\overline{D}_B$, $\underline{D}_B$. Indeed, \cite[Proposition~8.3]{BGSPredict} shows that Theorem~\ref{thm:SSOY2}.\ref{it:SSOY2 error bound}--\ref{it:SSOY2 deterministic} do not hold either for $\idim(\mu)$ or $D_H$, while Remark~\ref{rem:<>dim} indicates that Theorem~\ref{thm:SSOY2}.\ref{it SSOY2 lower error bound} is not true either for $\overline{D}_B$ or $\underline{D}_B$ (to see this, it is sufficient to consider any example of the measure $\mu$ with $D_H < k < \underline{D}_B$ for some $k \in \N$). On the other hand, if $\mu$ is exactly dimensional (see Subsection~\ref{subsec:dim}), then $\idim(\mu) = D_H$ (see \cite{Y82,FanLauRao02}).  
\end{rem}

We complete our results on the Schroer--Sauer--Ott--Yorke prediction error conjecture with a discussion on the relation between $h$ and $\delta$, presenting a non-trivial example for which the lower bounds in the conjecture do not hold in its alternative interpretation, even in the class of natural measures for smooth axiom~A diffeomorphisms of compact Riemannian manifolds. This issue is clarified in the following remark and the discussion afterwards. 

\begin{rem}\label{rem:order of quant}
The assertions of Theorems~\ref{thm:ssoy2.i hdim} and~\ref{thm:SSOY2} are valid for every $\delta > 0$ small enough depending on $h$, i.e.~for $\delta < \delta_0$ with $\delta_0 = \delta_0(h)$. An alternative interpretation of the Schroer--Sauer--Ott--Yorke prediction error conjecture leads to a question, whether the suitable estimates hold for typical $h$ with a fixed $\delta$. This is certainly true in the case of the upper bounds in assertions \ref{it: ssoy2 ii}--\ref{it: ssoy2 iii} of the conjecture, as for a given observable $h,$ if an upper bound holds for $\delta_0$, then it also holds for all $\delta > \delta_0$. The situation is different for the lower bound, for a trivial reason: if for a given observable $h$ we take $\delta$ much larger than $\diam (h(X))$, then $\sigma_{h,\eps}(y) \leq \delta$ for all $\eps$, and hence the set $\{ x \in X : \sigma_{h,\eps}(\phi_{h}(x)) > \delta \}$ is empty. This shows that the range $(0,\delta_0)$ for which a lower bound holds must depend on $h$.
\end{rem}

In view of the above remark, it is natural to consider non-constant observables and ask whether the lower bounds in assertions~\ref{it: ssoy2 i}--\ref{it: ssoy2 ii} of the prediction error conjecture hold for all $\delta < \delta_0$ with $\delta_0$ locally almost independent of $h$, i.e.~whether for a given non-constant observable $h_0$ there exists $\delta_0 > 0$ such that the lower bounds are valid for every $\delta < \delta_0$ and typical $h$ in some neighbourhood of $h_0$ (so that we can in particular consider $h$ with $\diam(h(X))$ uniformly bounded from below). The following theorem shows that this is not the case. In fact, it remains untrue even after setting $D=\hdim \mu$ or $D=\udim X$, so one cannot obtain the corresponding lower bounds in Theorem~\ref{thm:SSOY2} with $\delta_0$ locally almost independent of non-constant $h$.

\begin{thm}[{\bf Counterexample}{}]\label{thm:counterexample}
There exists a compact Riemannian manifold $M$, a $C^\infty$-axiom~A diffeomorphism $T\colon M \to M$ with a compact $T$-invariant attractor $X \subset M$ and a natural measure $\mu$ on $X$, such that $\idim(\mu) = \hdim \mu = \hdim X = \bdim X = D = 3/2$ and a non-constant Lipschitz observable $h_0 \colon X \to \R$, such that the following holds. For $k = 1,2$ and every $\delta > 0$, there exists $\eps_0> 0$ and an open set $\mathcal{U} \subset \Lip(X, \R)$ containing $h_0$ such that for every $h \in \mathcal{U}$, the corresponding $k$-delay coordinate map $\phi_{h}$ is not injective and for every $0 <\eps < \eps_0$ there holds
\[
\{ x \in X : \sigma_{h,\eps}(\phi_{h}(x)) > \delta \} = \emptyset.
\]
Consequently, for $k = 1$ we have $k < D$ and there is no $C = C(\delta, h) > 0$, such that
\[
\mu( \{ x \in X : \sigma_{h, \eps}(\phi_{h}(x)) > \delta  \}) \geq  C,
\]
so assertion~\ref{it: ssoy2 i} of the conjecture does not hold when the range of allowable $\delta$ is locally almost independent of $h$. Similarly, specifying to $k = 2$ we have $D < k < 2D$ and so the lower bound in assertion~\ref{it: ssoy2 ii} of the conjecture fails when the range of allowable $\delta$ is locally almost independent of $h$.
\end{thm}

The proof of Theorem~\ref{thm:counterexample} is presented in Section~\ref{sec:counterexample}. It remains an open problem whether assertion~\ref{it: ssoy2 ii} of the prediction error conjecture is true if the range of $\delta$ is allowed to depend on $h$.

\subsection{Dimension of conditional measures}\label{sec:cond meas dim}

To show Theorems~\ref{thm:excluding embed and predict}, \ref{thm:excluding embed and predict determ} and~\ref{thm:ssoy2.i hdim}, we prove more general results on the dimension of \emph{conditional measures} for the measure $\mu$ with respect to the delay coordinate map, which may be interesting on their own. Following \cite{SimmonsRohlin} (see also \cite[Subsection~2.4]{BGS22}), for a Borel map $\phi \colon X \to \R^k$ on a compact set $X \subset \R^N$ and a (complete) Borel probability measure $\mu$ on $X$, we define a system of measures $\mu_{\phi, y}$, $y \in \R^k$, where $\mu_{\phi, y}$ is a (possibly zero) Borel measure on $\phi^{-1}(y)$ defined as the weak-$^*$ limit 
\begin{equation}\label{eq:cond measure limit}
\mu_{\phi, y} = \lim_{\delta \to 0} \frac{1}{\mu(\phi^{-1}(B(y,\delta)))} \; \mu|_{\phi^{-1}(B(y,\delta))},
\end{equation}
whenever the limit exists, and zero otherwise. By the topological Rohlin disintegration theorem \cite{SimmonsRohlin}, the limit in \eqref{eq:cond measure limit} exists for $\phi \mu$-almost every $y \in \R^k$ and satisfies
\begin{equation}\label{eq:cond meas decomp}
\mu(E) = \int_{\R^k} \mu_{\phi, y}(E) \; d(\phi \mu)(y) \qquad\text{for  every }  \mu\text{-measurable } E \subset X
\end{equation}
(in particular, the function $\R^k \ni y \mapsto \mu_{\phi, y}(E)$ in \eqref{eq:cond meas decomp} is $\phi \mu$-measurable) and
\begin{equation}\label{eq:cond meas inverse}
\mu_{\phi, y}(\phi^{-1}(y)) = 1 \qquad\text{for $\phi \mu$-almost every } y \in \R^k.
\end{equation}
The system $\{\mu_{\phi, y}\}_{y \in \R^k}$ is called the \emph{system of conditional measures for $\mu$ with respect to $\phi$}. Moreover, the conditions \eqref{eq:cond meas decomp} and \eqref{eq:cond meas inverse} characterize the system  $\{\mu_{\phi, y}\}_{y \in \R^k}$ uniquely ($\phi \mu$-almost surely). See \cite{SimmonsRohlin} for details.

For a Borel transformation $T\colon X \to X$ on a Borel set $X \subset \R^N$ with a Borel probability measure $\mu$ on $X$ and a $k$-delay coordinate map $\phi_h\colon X \to \R^k$ corresponding to a Borel observable $h\colon X \to 
\R$, we consider a system of conditional measures $\{\mu_{h, y}\}_{y \in \R^k}$ for the (completion of) $\mu$ with respect to $\phi_h$, where for simplicity, we write $\mu_{h, y}$ instead of $\mu_{\phi_h, y}$. A direct relation between the almost sure injectivity of the delay coordinate maps, predictability of observables and the properties of the system of conditional measures is described in the following lemma.

\begin{lem}\label{lem: cond meas iff}Let $X \subset\R^N$, $N \in \N$, be a compact set, let $\mu$ be a Borel probability measure on $X$ and let $T\colon X \to X$ be a continuous map. Consider a continuous observable $h \colon X \to \R$ and fix $k \in \N$. Then the following hold.
\begin{enumerate}[$($a$)$]
	\item The $k$-delay coordinate map $\phi_h$ is almost surely injective if and only if $\mu_{h,y}$ is a Dirac's measure for $\phi_h \mu$-almost every $y \in \R^k$.
	\item The observable $h$ is almost surely $k$-predictable if and only if $\phi_h \circ T (\mu_{h,y})$ is a Dirac's measure for $\phi_h \mu$-almost every $y \in \R^k$.
	\item 
For $\phi_{h} \mu$-almost every $y \in \R^k$, the prediction error $\sigma_h(y) = \lim_{\eps\to 0}\sigma_{h,\eps}(y)$ exists and is equal to the standard deviation of a random variable with probability distribution $\phi_h \circ T (\mu_{h,y})$.
\end{enumerate}
\end{lem}

The proof of Lemma~\ref{lem: cond meas iff} is presented in Section~\ref{sec:cond measures imply}.

Our first result on the system of conditional measures is related to almost sure injectivity of delay coordinate maps and is a generalization of assertion~\ref{it:prob non injective} of Theorem~\ref{thm:excluding embed and predict}.

\begin{thm}\label{thm:injectivity sections}
	Let $X \subset\R^N$, $N \in \N$, be a compact set, let $T\colon X \to X$ be a Lipschitz map and let $\mu$ be a Borel probability measure on $X$. Fix $k \in \N$. Then, for a prevalent Lipschitz observable $h \colon X \to \R$ and the $k$-delay coordinate map $\phi_h$, the following hold.
	\begin{enumerate}[$($a$)$]
		\item\label{it:injectivity sections lhdim} $\lhdim T^{k-1}(\mu_{h, \phi_h(x)}) \geq \lhdim T^{k-1} \mu - k$ for $\mu$-almost every $x \in X$.
		\item\label{it:injectivity sections uhdim} For every $\eps > 0$, there holds $\hdim T^{k-1}(\mu_{h, \phi_h(x)}) \geq \hdim T^{k-1} \mu - k - \eps$ for $x$ from a set of positive $\mu$-measure.
	\end{enumerate}
\end{thm}

The second result is related to almost sure predictability and generalizes assertion~\ref{it:prob non predict} of Theorem~\ref{thm:excluding embed and predict}. 

\begin{thm}\label{thm:predictability sections}
	Let $X \subset\R^N$, $N \in \N$, be a compact set, let $T\colon X \to X$ be a Lipschitz map and let $\mu$ be a Borel probability measure on $X$. Fix $k \in \N$ and assume $\mu (  \Prep_k (T)) = 0$. Then, for a prevalent Lipschitz observable $h \colon X \to \R$ and the $k$-delay coordinate map $\phi_h$, the following hold.
	\begin{enumerate}[$($a$)$]
		\item\label{it:predictability sections lhdim} $\lhdim \, \phi_h \circ T(\mu_{h, \phi_h(x)}) \geq \min\{ 1, \lhdim T^{k} \mu - k\}$ for $\mu$-almost every $x \in X$.
		\item\label{it:predictability sections uhdim} For every $\eps > 0$, there holds $\hdim \phi_h \circ T(\mu_{h, \phi_h(x)}) \geq \min\{ 1, \hdim T^{k} \mu - k - \eps\}$ for $x$ from a set of positive $\mu$-measure.
	\end{enumerate}
\end{thm}

In Section~\ref{sec:cond measures imply} we show how Theorems~\ref{thm:injectivity sections}--\ref{thm:predictability sections} imply Theorems~\ref{thm:excluding embed and predict}, \ref{thm:excluding embed and predict determ} and~\ref{thm:ssoy2.i hdim}, while the proofs of Theorems~\ref{thm:injectivity sections}--\ref{thm:absolute continuity} are presented in Section~\ref{sec:proof sections}. As an important step in the proofs, in Section~\ref{sec:proof sections} we also show the following result, which might be of independent interest.

\begin{thm}\label{thm:absolute continuity}
Let $X \subset\R^N$, $N \in \N$, be a compact set, let $\mu$ be a Borel probability measure on $X$ and let $T\colon X \to X$ be a Lipschitz map. Fix $k \in \N$ such that $k < \lhdim T^{k-1} \mu$ and $($in the case $k > 1)$ assume $\mu ( \Prep_{k-1}(T) ) = 0$. Then for a prevalent Lipschitz observable $h \colon X \to \R$ and the $k$-delay coordinate map $\phi_h$, the measure $\phi_h \mu$ is absolutely continuous with respect to the $k$-dimensional Lebesgue measure in $\R^k$. 
If, additionally, $\mu$ is $T$-invariant or $T$ is bi-Lipschitz onto its image, then the condition $k < \lhdim T^{k-1}\mu$ may be replaced by $k < \lhdim \mu$ and the condition $\mu(\Prep_{k-1}(T)) = 0$ may be replaced by $\mu\big(\bigcup_{p=1}^{k-1} \Per_p(T)\big) = 0$.

\end{thm}

In the appendix to this paper, we prove an additional result (Theorem~\ref{thm:loc dim}) on the local dimensions of the measure $\phi_h \mu$, which extends \cite[Theorem~3.5]{SauerYorke97} (see also \cite[Remark~4.4]{SauerYorke97} for a version concerning delay coordinate maps) to a general setup of Lipschitz systems and non-invariant measures.

\subsection{Note on the proofs}

The proofs of Theorems~\ref{thm:injectivity sections}--\ref{thm:absolute continuity} are adaptations of the proofs of Marstrand--Mattila-type projection theorems from the classical theory of orthogonal projections (see \cite{mattila} for a comprehensive study) to the dynamical setting of delay coordinate maps. More specifically, the proofs of Theorems~\ref{thm:injectivity sections} and~\ref{thm:predictability sections} are inspired by the proofs of the slicing theorem for measures \cite[Theorem~3.3]{JMSections} (see also \cite[Theorem~10.7]{mattila}), while Theorem~\ref{thm:absolute continuity} is modelled after Marstrand--Mattila projection theorem for measures \cite[Theorem~6.1]{HuTaylorProjections} (see also \cite[Theorem~9.7]{mattila}). In the dynamical setting of delay coordinate maps, the two key elements of the proofs are establishing a correct version of \emph{transversality property} with respect to the parameter and treating carefully the obstructions arising from the existence of (pre)periodic points -- see Section~\ref{sec:technic} for details. Refer to \cite{SolomyakTransversSurvey} for an exposition of the concept of transversality in the context of orthogonal projections and iterated function systems. In particular, see \cite[Section 4.2]{SolomyakTransversSurvey} for an explanation of the terminology. A similar task of transferring Marstrand--Mattila projection theorem on dimension preservation \cite{Marstrand,Mattila-proj} to the setup of delay coordinate maps was performed by Sauer and Yorke in \cite{SauerYorke97} and we extend their result in Appendix~\ref{app:loc dim}.

\subsection{Structure of the paper} In Section~\ref{sec:prelim} we provide necessary definitions and results concerning, among others, several notions of dimensions of sets and measures used in this paper, as well as a discussion on the notion of prevalence. In Section~\ref{sec:cond measures imply} we
describe relations between conditional measures with respect to the delay coordinate map and injectivity/predictability properties of observables, 
proving Lemma~\ref{lem: cond meas iff} and showing how Theorems~\ref{thm:injectivity sections}--\ref{thm:predictability sections} imply Theorems~\ref{thm:excluding embed and predict}, \ref{thm:excluding embed and predict determ} and~\ref{thm:ssoy2.i hdim}. Section~\ref{sec:technic} introduces the key technical tools used in the proofs of the main results of the paper. More precisely, in Subsection~\ref{subsec:energy} we provide estimates of some `energy type' integrals, Subsection~\ref{subsec:matrices} contains basic facts on observation matrices that are used for checking prevalence of suitable observables, in Subsection~\ref{subsec:decomp} we describe a convenient decomposition of the phase space $X$, Subsection~\ref{subsec:restricted cond meas} considers `restricted' conditional measures, while Subsection~\ref{subsec:slices} introduces `geometric slices' of the measure and discusses their relation to the system of conditional measures. In Section~\ref{sec:proof sections} we prove the results on dimensions of conditional measures for the measure $\mu$ with respect to the delay coordinate map (Theorems~\ref{thm:injectivity sections}--\ref{thm:predictability sections}) together with Theorem~\ref{thm:absolute continuity}. In Section~\ref{sec:counterexample} we prove Theorem~\ref{thm:counterexample}, by providing suitable examples. Section~\ref{sec:examples} contains a discussion on the assumptions within the main results of paper, presenting several examples showing their necessity. Finally, in Appendix~\ref{app:loc dim}, we prove results on the local dimensions of the push-forward of the measure $\mu$ by the delay coordinate maps (Theorem~\ref{thm:loc dim}, Corollary~\ref{cor:loc dim} and Theorem~\ref{thm:loc dim full}).

\section{Preliminaries}\label{sec:prelim}

By $\N$ we denote the set of positive integers. The symbols $\| \cdot \|$, $\dist(\cdot, \cdot)$ and $| \cdot |$ denote, respectively, the Euclidean norm, distance and diameter in $\R^N$, $N \in \N$. We set $\|(x_1, \ldots, x_N)\|_{\infty} = \max\{|x_1|,  \ldots, |x_N|\}$ for $(x_1, \ldots, x_N) \in \R^N$. The open $r$-ball around a point $x \in \R^N$ is denoted by $B_N(x, r)$ and the closed ball by $\overline{B_N}(x, r)$ (sometimes we omit the dimension $N$ in notation). By $\Leb$ we denote the Lebesgue measure (or the outer Lebesgue measure, in the case of non-measurable sets). 

\subsection{Singular values}\label{subsec:sing_values}

Let $\psi \colon \R^m \to \R^k$ be a linear map and let $A$ be the matrix of $\psi$. For $p \in \{1, \ldots, k\}$ let $\sigma_p(A)$ (we also use the symbol $\sigma_p(\psi)$) be the $p$th largest \emph{singular value} of $A$, i.e.~the $p$th largest square root of an eigenvalue of the matrix $A^*A$ (counted with multiplicities).  It is well-known (see e.g.~\cite [Lemma 14.2]{Rob11}) that the the rank of $A$ equals the number of the non-zero singular values of $A$.

For $k \le m$, the matrix $A$ admits the \emph{singular value decomposition}  in the form $A = U\Sigma V^T$, where $U$ is a $k \times k$ orthogonal matrix, $V$ is an $m \times m$ orthogonal matrix, and $\Sigma$ is the $k \times m$ rectangular diagonal matrix of the form
\[
\Sigma = 
\left[\begin{array}{@{}c|c@{}}
  \begin{matrix}
  \sigma_1(A) & &0\\
   & \ddots &\\
   0& & \sigma_k(A)
  \end{matrix}
  &\text{\; \LARGE $0$\;} \\
\end{array}\right].
\]

We will use the following lemma, proved as \cite[Lemma~4.2]{SYC91} (see also \cite[Lemma~14.3]{Rob11}).

\begin{lem}\label{lem: key_ineq_inter}
Let $\psi \colon  \R^m \to \R^k$ be a linear transformation. Assume $\sigma_p(\psi) > 0$ for some $p \in \{1, \ldots, k\}$. Then for every $z \in \R^k$ and $\rho, \eps > 0$,
\[ \frac{\Leb(\{ \alpha \in B_m(0, \rho) : \|\psi(\alpha) + z \| \leq \eps \})}{\Leb (B_m(0, \rho))} \leq C \Big(\frac{\eps}{\sigma_p(\psi) \, \rho}\Big)^p, \]
where $C > 0$ depends only on $m,k$.
\end{lem}

\subsection{Measures}\label{subsec:meas}

Let $\mu$ be a Borel measure on a Borel set $X \subset \R^N$. A Borel set $Y \subset X$ is called a \emph{full-measure} subset of $X$, if $\mu (X \setminus Y) = 0$. By $\supp \mu$ we denote the (topological) \emph{support} of $\mu$, which is the smallest closed subset of full $\mu$-measure. \emph{Dirac's measure} at a point $x$, denoted by $\delta_x$, is the Borel probability measure supported on $\{x\}$. For Borel measures $\mu, \nu$ in $\R^N$ we write $\nu \ll \mu$, if $\nu$ is \emph{absolutely continuous} with respect to $\mu$, i.e.~if $\nu(E) = 0$ whenever $\mu(E) = 0$ for Borel sets $E \subset \R^N$. 
For a Borel map $\phi\colon X \to \R^k$, $k \in \N$, by $\phi\mu$ we denote the \emph{push-forward} of $\mu$ under $\phi$, defined by $\phi\mu(E) = \mu(\phi^{-1}(E))$ for Borel sets $E \subset \R^k$. 

Let $T\colon X \to X$ be a Borel map. The measure $\mu$ is $T$-\emph{invariant}, if $\mu(T^{-1}(E)) = \mu(E)$ for every Borel set $E \subset X$. The measure $\mu$ is \emph{ergodic}, if for every Borel set $E \subset X$ such that $T^{-1}(E) = E$, there holds $\mu(E) = 0$ or $\mu(X \setminus E) = 0$.

\begin{defn}[{\bf Natural measure}{}]\label{defn:nat_measure}
Let $M$ be a compact Riemannian manifold and $T \colon M \to M$ a $C^1$-diffeomorphism. A compact $T$-invariant set $X \subset M$ is called an \emph{attractor}, if there exists an open set $B \subset M$ containing $X$, such that $\lim_{n \to \infty} \dist(T^n x, X) = 0$ for every $x \in B$. The largest such set $B(X)$ is called the (\emph{maximal}) \emph{basin of attraction} of $X$. A $T$-invariant Borel probability measure $\mu$ on $X$ is called a \emph{natural measure} if 
	\[ \lim_{n \to \infty} \frac{1}{n} \sum_{i=0}^{n-1} \delta_{T^i x} = \mu \]
	for almost every $x \in B(X)$ with respect to the volume measure on $M$, where   the limit is taken in the weak-$^*$ topology.
\end{defn}

\subsection{Prevalence}\label{subsec:preval} 
As noted in the introduction, we understand the genericity in the space of Lipschitz observables on compact sets in the sense of prevalence -- a notion introduced by Hunt, Sauer and Yorke in \cite{Prevalence92}, which may be considered as an analogue of `Lebesgue almost sure' condition in infinite dimensional normed linear spaces. 

\begin{defn}\label{defn:preval} By a (complete) \emph{linear metric} on a linear space we mean a (complete) metric which makes addition and scalar multiplication continuous. Let $V$ be a complete linear metric space (i.e.~a linear space with a complete linear metric). A Borel set $\mathcal S \subset V$ is called \emph{prevalent} if there exists a Borel measure $\nu$ in $V$, which is positive and finite on some compact set in $V$, such that for every $v \in V$, there holds $v + e \in \mathcal S$ for $\nu$-almost every $e \in V$. A non-Borel subset of $V$ is prevalent if it contains a prevalent Borel subset. For more information on prevalence we refer to \cite{Prevalence92} and \cite[Chapter~5]{Rob11}.
\end{defn}

For a compact set $X \subset \R^N$, $N \in \N$ we consider the space $\Lip(X)$ of Lipschitz functions on $X$. Recall that a function $h\colon X \to \R$ is \emph{Lipschitz} if there exists $L > 0$ such that $|h(x) - h(y)| \leq L \|x-y\|$ for every $x, y \in X$. We set 
\[
\|h \|_{\Lip(X)} = \sup_{x \in X} |h(x)| + \Lip(h), \quad \text{where } \Lip(h) = \sup_{\substack{x, y \in X \\ x \neq y}} \frac{|h(x) - h(y)|}{\|x-y\|}
\]
to be the \emph{Lipschitz norm} on $\Lip(X)$. With this norm, the space $\Lip(X)$ is a Banach space (in particular, a complete linear metric space). In this paper we use the notion of prevalence in the sense of Definition~\ref{defn:preval} applied to $V = \Lip(X)$. In fact, similarly as in \cite{SYC91, Rob11, BGS22, BGSPredict}, we as the measure $\nu$ in Definition~\ref{defn:preval} we employ $\nu = \xi\Leb$ for $\xi \colon \R^m \to \Lip(X)$, $\xi(\alpha_1, \ldots, \alpha_m) = \sum_{j=1}^m \alpha_j h_j$, where $\{h_1, \ldots, h_m\}$ (called the \emph{probe set}) is the set of all real monomials of $N$ variables of degree at most $d$ for $d = 2k+1$ (in fact, most of the results proved in this paper hold also for $d =2k-1$), and $\Leb$ is the $k$-dimensional Lebesgue measure in $\R^k$.
Then a set $\mathcal S \subset \Lip(X)$ is prevalent if for every $h \in\Lip(X)$, the function $h + \sum_{j=1}^m \alpha_j h_j$ is in $\mathcal S$ for Lebesgue-almost every $(\alpha_1, \ldots, \alpha_m) \in \R^m$. Note that whenever prevalence in $\Lip(X)$ is established via a probe set consisting of polynomials as above, it also provides prevalence in the spaces $C^r(X)$, $r=1,2,\ldots, \infty$.

\subsection{Dimensions}\label{subsec:dim} For convenience, we present the definitions of all notions of dimensions that appear in this paper.

\begin{defn}\label{defn:dim}
	
	For $s>0$, the \emph{$s$-dimensional $($outer$)$ Hausdorff measure} of a set $X \subset \R^N$ is defined  as
	\[ \mH^s(X) = \lim_{\delta \to 0}\ \inf \Big\{ \sum_{i = 1}^{\infty} |U_i|^s : X \subset \bigcup_{i=1}^{\infty} U_i,\ |U_i| \leq \delta  \Big\}.\]
	The \emph{Hausdorff dimension} of $X$ is given as
	\[ \hdim X = \inf \{ s > 0 : \mathcal{H}^s(X) = 0 \} = \sup \{ s > 0 : \mathcal{H}^s(X) = \infty \} \]
with the convention $\sup \emptyset = 0$. 	
\end{defn}

\begin{defn}\label{defn:bdim}
For a bounded set $X \subset \R^N$ and $\delta>0$, let $N(X, \delta)$ denote the minimal number of balls of diameter at most $\delta$ required to cover $X$. The \emph{lower} and \emph{upper box-counting $($Minkowski$)$ dimensions} of $X$ are defined, respectively, as
\[ \ldim_B X = \liminf_{\delta \to 0} \frac{\log N(X,\delta)}{-\log \delta},\qquad \udim X = \limsup_{\delta \to 0} \frac{\log N(X,\delta)}{-\log \delta}. \]
If $\ldim_B X = \udim X$, then their common value is denoted by $\bdim X$ and called the \emph{box dimension} of $X$.
\end{defn}

\begin{defn}
Let $\mu$ be a finite Borel measure on $\R^N$. The \emph{lower} and \emph{upper local dimensions} of $\mu$ at a point $x \in  \supp \mu$ are defined, respectively, as
	\[ \underline{d}(\mu, x) = \liminf_{\delta \to 0} \frac{\log \mu(B(x,\delta))}{\log \delta},\qquad  \overline{d}(\mu, x) = \limsup_{\delta \to 0} \frac{\log \mu(B(x,\delta))}{\log \delta}.\]
If $\underline{d}(\mu, x) = \overline{d}(\mu, x)$, then their common value is denoted by $d(\mu, x)$ and called the \emph{local dimension} of $\mu$ at $x$. If $d(\mu, x)$ exists and equals to some $d$ for $\mu$-almost every $x$, then we say that $\mu$ is \emph{exact dimensional} with dimension $d$.
	
	The \emph{upper} and \emph{lower Hausdorff dimension} of $\mu$ are defined, respectively,  as 
\begin{align*}
\hdim \mu &= \inf \{ \hdim E: \mu(\R^N \setminus E) = 0 \} = \underset{x \sim \mu}{\mathrm{ess\ sup}}\ \underline{d}(\mu, x),\\
\lhdim \mu &= \inf \{ \hdim E: \mu(E) > 0 \} = \underset{x \sim \mu}{\mathrm{ess\ inf}}\ \underline{d}(\mu, x).
\end{align*}
See \cite[Propositions 10.2--10.3]{FalconerTechniques} for the proof of the equivalence of both variants of the definitions. 
\end{defn}

The definitions immediately imply that for finite Borel measures $\mu$, $\nu$ on $\R^N$ the following holds.
\[
\text{If } \;  \nu \ll \mu, \; \text{ then } \; \hdim \nu \leq \hdim \mu \; \text{ and } \; \lhdim \nu \geq \lhdim \mu.
\]

\begin{defn}
For a Borel probability measure $\mu$ in $\R^N$ with compact support, its \emph{lower} and \emph{upper information dimensions} are
\[ \lid(\mu) = \liminf_{\eps \to 0} \int_{\supp \mu} \frac{\log \mu(B(x,\eps))}{\log \eps} d\mu(x), \qquad\uid(\mu) = \limsup_{\eps \to 0} \int_{\supp \mu} \frac{\log \mu(B(x,\eps))}{\log \eps} d\mu(x).\]
If $\lid(\mu) = \uid(\mu)$, then we denote their common value as $\idim (\mu)$ and call it the \emph{information dimension} of $\mu$. 
\end{defn}

\begin{defn}\label{defn:energy}
	For a finite Borel measure $\mu$ in $\R^N$ and $s>0$, the $s$-\emph{potential} of $\mu$ at a point $x \in \R^N$ is defined as
	\[ \mE_s(\mu, x) = \int_{\R^N}  \frac{d\mu(y)}{\|x - y\|^s}, \]
while the $s$-\emph{energy} of $\mu$ is 
	\[ \mE_s(\mu) = \int_{\R^N} \mE_s(\mu, x) d\mu(x) =  \int_{\R^N}\int_{\R^N}  \frac{d\mu(x)d\mu(y)}{\|x - y\|^s}.\]
	The \emph{correlation dimension} of $\mu$ is defined as
	\[\cdim \mu = \sup \{ s > 0 : \mE_s(\mu) < \infty \},\]
	with the convention $\sup \emptyset = 0$. 
For an explanation of this terminology see \cite[$\S$2]{SauerYorke97} and \cite[Section 17]{pesin2008dimension}. 
Note that if $\mu$ has an atom, then $\cdim \mu = 0$. It turns out that the lower local dimension can be characterized in terms of the local potentials as follows (see \cite[Section 3.2]{SauerYorke97} and \cite[Section 4]{HuntKaloshin97}).
	
	\end{defn}

\begin{lem}\label{lem:loc dim energy} Let $\mu$ be a finite Borel measure in $\R^N$. Then for every $x \in \supp \mu$,
\[ \ld(\mu, x) = \sup \{ s>0 : \mE_s(\mu, x) < \infty \}.\]
\end{lem}

\begin{rem}
If $\mu$ is a finite Borel measure on a bounded Borel set $X$ in $\R^N$, then
\[ \lhdim \mu \leq \hdim \mu \leq \hdim X \leq \ldim_B X \leq \udim X \]
and
\[ \lid(\mu)  \leq \ldim_B X, \qquad \uid(\mu) \leq \udim X. \]
In general, there are no relations between Hausdorff and information dimensions. However, it is known that if $\mu$ is a $T$-invariant ergodic measure for a Lipschitz map $T\colon X \to X$, then $\hdim \mu \leq \lid(\mu) \leq \uid(\mu)$ (see \cite[Proposition~2.1]{BGS22}).
\end{rem}

The relation between the Hausdorff and correlation dimension is not so direct, but one has the following result (see e.g. \cite[Theorem~8.7 and Remark~8.6]{mattila}).

\begin{lem}\label{lem:corr to Hausdorff}
Let $E \subset \R^N$ be a Borel set. If $\mu$ is a finite Borel measure on $E$ satisfying $\mE_s(\mu) < \infty$ for some $s>0$, then $\hdim E \geq s$. Consequently, $\lhdim \mu \geq \cdim \mu$.
\end{lem}

For more information on dimension theory in Euclidean spaces see \cite{falconer2014fractal, mattila, Rob11}.

\section{Relations between conditional measures and injectivity/predictability properties} \label{sec:cond measures imply}

In this section we prove Lemma~\ref{lem: cond meas iff} and explain how Theorems~\ref{thm:excluding embed and predict}, \ref{thm:excluding embed and predict determ} and \ref{thm:ssoy2.i hdim} follow from Theorems~\ref{thm:injectivity sections}--\ref{thm:predictability sections}.

\begin{proof}[Proof of Lemma~\rm\ref{lem: cond meas iff}] To show assertion~(a), suppose first that $\mu_{h,y}$ is a Dirac's measure for $\phi_h\mu$-almost every $y \in \R^k$. This means that for $\phi_h\mu$-almost every $y \in \R^k$ there exists a (unique) $f(y) \in X$ such that $\mu_{h,y} = \delta_{f(y)}$. Then by \eqref{eq:cond meas decomp}, for every $\mu$-measurable set $E \subset X$, the set $\{y \in \R^k: f(y) \in E\}$ is $\phi_h\mu$-measurable, so $f$ is  measurable. Therefore the set 
\[
\tilde X_h = \{x \in X: f(\phi_h(x)) = x\}
\]
is $\mu$-measurable. Note that by \eqref{eq:cond meas inverse}, we have $f(y) \in \phi_h^{-1}(y)$ for $\phi_h \mu$-a.e. $y$, hence also  $\phi_h(f({\phi_h(x)})) = \phi_h(x)$ for $\mu$-almost every $x \in X$. Applying $f$ to both sides of      the last equality yields $f({\phi_h(x)}) \in \tilde X_h$ for $\mu$-almost every $x \in X$. Consequently, \eqref{eq:cond meas decomp} gives
\[\begin{aligned}
\mu(\tilde X_h) & = \int_{\R^k} \mu_{h,y}(\tilde X_h) d(\phi_h\mu)(y) =  \int_{\R^k} \delta_{f(y)}(\tilde X_h) d(\phi_h\mu)(y) \\
& =  \int_X \delta_{f(\phi_h(x))}(\tilde X_h) d\mu(x) = \int_X 1 d\mu(x) = 1,
\end{aligned}
\]
so $\tilde X_h$ is a full $\mu$-measure set. If $x,y \in \tilde X_h$ and $\phi_h(x) = \phi_h(y)$, then $x = f(\phi_h(x)) = f(\phi_h(y)) = y$, hence $\phi_h$ is injective on $\tilde X_h$. By the regularity of $\mu$, the set $\tilde X_h$ contains a full $\mu$-measure Borel subset $X_h$. This shows that $\phi_h$ is almost surely injective.  

Conversely, if $\phi_h$ is injective on a Borel set $X_h \subset X$ of full $\mu$-measure, then a system of measures $\{\tilde\mu_{h,y}\}_{y \in \R^k}$ given by
\[
\tilde\mu_{h,y} = 
\begin{cases}
\delta_{(\phi_h|_{X_h})^{-1}(y)} &\text{for } y \in \phi_h(X_h)\\
0 &\text{otherwise}
\end{cases}
\]
satisfies the conditions \eqref{eq:cond meas decomp}--\eqref{eq:cond meas inverse}, so it is a system of conditional measures of $\mu$ with respect to $\phi_h$. Hence, by the uniqueness of the system of conditional measures, $\mu_{h,y} = \tilde\mu_{h,y}$ is a Dirac's measure for $\phi_h\mu$-almost every $y \in \R^k$. The details of this argument together with a precise discussion on measurability issues are presented in \cite[Proof of Theorem~3.1]{BGS22}.

Assertion~(c) is proved as \cite[Lemma~3.2]{BGS22}. An immediate consequence of (c) is that for $\phi_h \mu$-almost every $y \in \R^k$, the point $y$ is $k$-predictable (i.e. $\sigma_h(y) = 0$) if and only if the measure $\phi_h \circ T (\mu_{h,y})$ is a Dirac's measure. As pointed out in Remark~\ref{rem:two kinds of predict}, this condition is also equivalent to the almost sure $k$-predictability of $h$, which proves assertion~(b).
\end{proof}

Now, supposing Theorems~\ref{thm:injectivity sections}--\ref{thm:predictability sections} are true, we show how they imply Theorems~\ref{thm:excluding embed and predict}, \ref{thm:excluding embed and predict determ} and~\ref{thm:ssoy2.i hdim}. We start by proving Theorem~\ref{thm:ssoy2.i hdim}.

\begin{proof}[{Proof of Theorem~\rm\ref{thm:ssoy2.i hdim}}]

By assumption, $\hdim T^k(\mu|_{X \setminus \Prep_k(T)}) > k > 0$, which implies $\mu(X \setminus \Prep_k(T)) > 0$. Let 
\[
\tilde\mu = \frac{1}{\mu(X \setminus \Prep_k(T))}\mu|_{X \setminus \Prep_k(T)}.
\]
Then $\tilde\mu$ is a Borel probability measure on $X$ such that $\tilde\mu(\Prep_k(T)) = 0$ and $\hdim T^k\tilde \mu > k$. Hence, by Theorem~\ref{thm:predictability sections}\ref{it:predictability sections uhdim} applied for $\tilde\mu$, for a prevalent Lipschitz observable $h\colon X \to \R$ we have $\hdim \phi_h \circ T(\mu_{h, \phi_h(x)}) > 0$ for every $x$ from a $\tilde\mu$-positive measure set, hence also for every $x$ from a $\mu$-positive measure set. Consequently, $\phi_h \circ T(\mu_{h, \phi_h(x)})$ is not a Dirac's measure (and consequently, a random variable with probability distribution $\phi_h \circ T(\mu_{h, \phi_h(x)})$ has positive standard deviation) for every $x$ from a $\mu$-positive measure set. By Lemma~\ref{lem: cond meas iff}(c), this implies that for a prevalent $h$ there exists $\delta_0 > 0$ such that 
\[
\lim_{\eps \to 0} \sigma_{h,\eps}(\phi_h(x)) = \sigma_h(\phi_h(x)) > \delta_0
\]
for every $x \in Y_h$, where $Y_h \subset X$ is a set of positive $\mu$-measure. Therefore, for every $0 < \delta <\delta_0$,
\[ 
\mu(\{x \in Y_h: \sigma_{h,\eps}(\phi_h(x)) \le \delta\}) \to 0 \quad \text{as } \eps \to 0, 
\]
so
\[ 
\mu(\{x \in Y_h: \sigma_{h,\eps}(\phi_h(x)) > \delta\}) \to \mu(Y_h) > 0 \quad \text{as } \eps \to 0.
\]
This shows the main assertion of Theorem~\ref{thm:ssoy2.i hdim}. 

To prove the additional ones note first that if $T$ is bi-Lipschitz onto its image, then by the definition of the Hausdorff dimension of a measure, $\hdim T^k(\mu|_{X \setminus \Prep_k(T)})= \hdim \mu|_{X \setminus \Prep_k(T)}$. Moreover, $T$ is injective in this case, so every pre-periodic point is actually periodic and $\Prep_k(T) =  \bigcup_{p=1}^k \Per_p(T)$, so $\hdim \mu|_{X \setminus \Prep_k(T)} = \hdim \mu|_{X \setminus \bigcup_{p=1}^k \Per_p(T)}$. Consequently, it suffices to assume $k < \hdim \mu|_{X \setminus \bigcup_{p=1}^k \Per_p(T)}$ instead of $k < \hdim T^k(\mu|_{X \setminus \Prep_k(T)})$

Suppose now $\mu$ is $T$-invariant. Then 
\[
\mu \Big(T^{-k}\Big(\bigcup_{p=1}^k \Per_p(T)\Big) \setminus \bigcup_{p=1}^k \Per_p(T)\Big)=0.
\]
Hence, using the fact $\bigcup_{p=1}^k \Per_p(T) \subset \Prep_k(T) \subset T^{-k}(\bigcup_{p=1}^k\Per_p(T))$, we obtain 
\[
\mu\Big(T^{-k}\Big(\bigcup_{p=1}^k\Per_p(T)\Big) \setminus \Prep_k(T)\Big) = 0
\]
and, consequently,
\begin{multline*}
T^k(\mu|_{X \setminus \Prep_k(T)})(E) = \mu (T^{-k}(E) \setminus\Prep_k(T))\\
= \mu \Big(T^{-k}(E) \setminus T^{-k}\Big(\bigcup_{p=1}^k \Per_p(T)\Big)\Big) +  \mu \Big(T^{-k}(E) \cap T^{-k}\Big(\bigcup_{p=1}^k \Per_p(T)\Big) \setminus \Prep_k(T) \Big)\\
=\mu \Big(T^{-k}(E) \setminus T^{-k}\Big(\bigcup_{p=1}^k \Per_p(T)\Big)\Big) = 
\mu \Big(E \setminus \bigcup_{p=1}^k \Per_p(T)\Big) = \mu|_{X \setminus \bigcup_{p=1}^k \Per_p(T)}(E)
\end{multline*}
for Borel sets $E$. This implies that if $\mu$ is $T$-invariant, then $T^k(\mu|_{X \setminus \Prep_k(T)}) = \mu|_{X \setminus \bigcup_{p=1}^k \Per_p(T)}$, which gives the required assertion.

Suppose now $\mu$ is $T$-invariant and ergodic. If $\mu(\bigcup_{p=1}^k \Per_p(T)) = 0$, then $T^k(\mu|_{X \setminus \Prep_k(T)}) = \mu|_{X \setminus \bigcup_{p=1}^k \Per_p(T)} = \mu$, which ends the proof. Otherwise, $\mu$ is supported on the periodic orbit of $x$ (see e.g. \cite[Remark 4.4]{BGS20}. The proof therein is for injective $T$, but it can be modified in a straightforward manner to the general case). Then $\hdim \mu = 0 = \hdim T^k(\mu|_{X \setminus \Prep_k(T)})$, hence the proof is finished.
\end{proof}

\begin{proof}[{Proof of Theorem~\rm\ref{thm:excluding embed and predict}}] 
Suppose $k < \hdim T^{k-1}\mu$. Then by Theorem~\ref{thm:injectivity sections}\ref{it:injectivity sections uhdim}, we have $$\hdim T^{k-1} (\mu_{h, \phi_h(x)}) > 0$$ for $x$ from a positive $\mu$-measure set, so $T^{k-1} (\mu_{h, \phi_h(x)})$ (and hence $\mu_{h, \phi_h(x)}$) is not a Dirac measure for $x$ from a positive $\mu$-measure set. Consequently, $\mu_{h, y}$ is not a Dirac measure for $y$ from a positive $\phi_h\mu$-measure set. By Lemma~\ref{lem: cond meas iff}(a), this implies the first part of assertion~\ref{it:prob non injective}. If, additionally, $\mu$ is $T$-invariant or $T$ is bi-Lipschitz onto its image, then $\hdim T^k\mu= \hdim \mu$, which shows the second part of assertion~\ref{it:prob non injective}.

To show assertion~\ref{it:prob non predict}, note that if $h$ is almost surely $k$-predictable, then $\sigma_{h,\eps}(\phi_h(x)) \to 0$ as $\eps \to 0$ for $\mu$-almost every $x \in X$ (see Remark~\ref{rem:two kinds of predict}), so $\mu(\{x \in X: \sigma_{h,\eps}(\phi_h(x)) > \delta\}) \to 0$ as $\eps \to 0$ for every $\delta > 0$. Therefore, assertion~\ref{it:prob non predict} follows directly from Theorem~\ref{thm:ssoy2.i hdim}.
\end{proof}

\begin{proof}[{Proof of Theorem~\rm\ref{thm:excluding embed and predict determ}}] In view of Theorem~\ref{thm:excluding embed and predict}\ref{it:prob non injective}, to show assertion~\ref{it:determ non injective} it is sufficient to construct a Borel probability measure $\mu$ on $X$ with $\hdim T^{k-1}\mu > k$. To do it, note that since $\hdim T^{k-1}(X) > k$, Frostman's lemma (see e.g.~\cite[Theorem~8.8]{mattila}) implies that there exists a Borel probability measure $\nu$ on $T^{k-1}(X)$ with $\hdim \nu > k$. As $T$ is continuous, the Kuratowski--Ryll--Nardzewski selection theorem \cite[Theorem~12.13]{K95} ensures that there exists a Borel partial inverse to $T^{k-1}$, i.e.~a Borel map $F \colon T^{k-1}(X) \to X$ satisfying $T^{k-1}(F y) = y$ for $y \in T^{k-1}(X)$ (to apply the theorem, we use the fact that a continuous image of an open subset of $X$ is Borel, as a countable union of compact sets). Set $\mu = F \nu$. Then $\mu$ is a Borel probability measure on $X$ such that $T^{k-1} \mu = \nu$. Hence, $\hdim T^{k-1} \mu = \hdim \nu > k$. This shows assertion~\ref{it:determ non injective}.

For assertion~\ref{it:determ non predict}, using Theorem~\ref{thm:excluding embed and predict}\ref{it:prob non predict}, it is sufficient to construct a Borel probability measure $\mu$ on $X \setminus \Prep_k(T)$ with $\hdim T^k\mu > k$. Similarly as previously, Frostman's lemma implies the existence of a Borel probability measure $\nu$ on $T^k(X) \setminus \bigcup_{p=1}^k\Per_p(T)$ with $\hdim \nu > k$. Taking $F\colon T^k(X) \setminus \bigcup_{p=1}^k\Per_p(T) \to X$ to be a Borel partial inverse to $T^k$ and setting $\mu = F \nu$, we obtain $\hdim T^k\mu > k$. To end the proof, it is enough to notice that $\bigcup_{p=1}^k\Per_p(T) = T^k(\Prep_k(T))$, so $F\big(T^k(X) \setminus \bigcup_{p=1}^k\Per_p(T)\big) \subset X \setminus \Prep_k(T)$, and hence $\mu$ is supported on $X \setminus \Prep_k(T)$.
\end{proof}

\section{Technical tools}\label{sec:technic}

\subsection{Energy integral estimates}\label{subsec:energy}
The following lemma can be proved by the same arguments as \cite[Lemma 2.6]{SauerYorke97}. For the reader's convenience, we include a complete proof.
	
	\begin{lem}\label{lem:energy_int_ineq}
		Let $A$ be the matrix of a linear transformation $\psi\colon \R^m \to \R^k$, $m, k \in \N$, $m\ge k$ and let $p \in \{1, \ldots , k\}$ be such that the $p$th largest singular value $\sigma_p(A)$ is positive. Then for every $b \in \R^k$ and $0<s<p$,
		\begin{equation}\label{eq:int singular ineq} \int_{B_m(0,1)} \frac{d\alpha}{\|A\alpha + b\|^s} \leq \frac{C}{(\sigma_p(A))^s},
		\end{equation}
		where $C > 0$ depends only on $m, k, p$ and $s$.
	\end{lem}
	
	\begin{proof}
		First, we prove \eqref{eq:int singular ineq} for $b = 0$. Let $A = U\Sigma V^T$ be the singular value decomposition of $A$ (see Subsection~\ref{subsec:sing_values}). As $U$ and $V$ are orthogonal, we have, for $\alpha = (\alpha_1, \ldots, \alpha_m)$,
		\begin{equation}\label{eq:b=0}
			\begin{split} \int_{B_m(0,1)} \frac{d\alpha}{\|A\alpha\|^s} & = \int_{B_m(0,1)} \frac{d\alpha}{\|\Sigma\alpha\|^s} = \int_{B_m(0,1)} \bigg(\sum_{j=1}^k (\sigma_j(A))^2 \alpha_j^2\bigg)^{-s/2} d\alpha_1 \cdots d \alpha_m \\
				& \leq \int_{B_m(0,1)} \bigg(\sum_{j=1}^p (\sigma_p(A))^2 \alpha_j^2\bigg)^{-s/2} d\alpha_1 \cdots d \alpha_m = \frac{c}{(\sigma_p(A))^s}
			\end{split}
		\end{equation}
		for
		\[      
		c = \int_{B_m(0,1)} \bigg(\sum_{j=1}^p \alpha_j^2\bigg)^{-s/2} d\alpha_1 \cdots d \alpha_m 
		\]
This yields \eqref{eq:int singular ineq} with $b = 0$. To prove \eqref{eq:int singular ineq} for an arbitrary $b \in \R^k$, consider first the case when $\|b\| \geq 2\|A\|$, where $\|A\|$ denotes the matrix norm. Then $\|A\alpha + b\| \geq \|A\alpha\|$ for $\alpha \in B_m(0,1)$, so by \eqref{eq:b=0},
		\[ 
		\int_{B_m(0,1)} \frac{d\alpha}{\|A\alpha + b\|^s}  \leq \int_{B_m(0,1)} \frac{d\alpha}{\|A\alpha\|^s} \le \frac{c}{(\sigma_p(A))^s},
		\]
		providing \eqref{eq:int singular ineq} in this case. For the remaining case, let $ \|b\| < 2 \|A\|$, write $b = b_1 + b_2$, where $b_1 \in \Im A$ and $b_2 \in (\Im A)^\perp$. Then $\|b_1\| \le \|b\| < 2 \|A\|$, so there exists $\alpha_0 \in B_m(0,2)$ such that $b_1 = A \alpha_0$. Therefore,
		\begin{equation*}
			\begin{split}
				\int_{B_m(0,1)} \frac{d\alpha}{\|A\alpha + b\|^s}  & = \int_{B_m(0,1)}  \frac{d\alpha}{(\|A\alpha + b_1\|^2 + \|b_2\|^2)^{s/2}} \leq \int_{B_m(0,1)}  \frac{d\alpha}{\|A\alpha + b_1\|^s}\\
				& = \int_{B_m(0,1)} \frac{d\alpha}{\|A(\alpha + \alpha_0)\|^s}  \leq \int_{B_m(0,3)} \frac{d\alpha}{\|A \alpha \|^s}\\
				& = 3^{m-s} \int_{B_m(0,1)} \frac{d\alpha}{\|A\alpha\|^s} \le \frac{c 3^{m-s}}{(\sigma_p(A))^s},
			\end{split}
		\end{equation*}
		where the two last steps follow, respectively, by the change of variables $\alpha \mapsto \alpha/3$ and using \eqref{eq:b=0}. This finishes the proof.
	\end{proof}

\subsection{Observation matrices}\label{subsec:matrices}

Let $X \subset \R^N$, $N \in \N$, be a compact set and let $T\colon X \to X$ be a Lipschitz transformation. Fix $k \in \N$, $d \ge 2k - 1$, and let $\{h_1, \ldots, h_m\}$ be the set of all real monomials of $N$ variables of degree at most $d$. Note that $m \ge k$. For a Lipschitz observable $h \colon X \to \R$ and $\alpha = (\alpha_1, \ldots, \alpha_m) \in \R^m$ let $h_\alpha \colon X \to \R$ be given by 
\[
h_\alpha = h + \sum_{j=1}^m \alpha_j h_j.
\]
For simplicity, we write $\phi_\alpha$ instead of $\phi_{h_\alpha}$ for the $k$-delay coordinate map corresponding to $h_\alpha$, i.e.
\[
\phi_{\alpha} \colon X \to \R^k, \qquad \phi_\alpha(x) = (h_\alpha(x), h_\alpha(Tx), \ldots, h_\alpha(T^{k-1}x)). \]
Note that for $x,y \in X$ we have
\begin{equation}\label{eq:matrix_form} \phi_\alpha(x) - \phi_\alpha(y) = D_{x,y}\alpha + w_{x,y}
\end{equation}
for a $k\times m$ matrix $D_{x,y}$ defined by
\begin{equation}\label{eq:D_xy}
	D_{x,y} = \begin{bmatrix} h_1(x) - h_1(y) & \ldots & h_m(x) - h_m(y) \\
		h_1(Tx) - h_1(Ty) & \ldots & h_m(Tx) - h_m(Ty) \\
		\vdots & \ddots & \vdots \\
		h_1(T^{k-1}x) - h_1(T^{k-1}y) & \ldots & h_m(T^{k-1}x) - h_m(T^{k-1}y) \\
	\end{bmatrix} 
\end{equation}
and
\[
w_{x,y} = \begin{bmatrix} h(x) - h(y) \\
	h(Tx) - h(Ty)\\
	\vdots \\
	h(T^{k-1} x) - h(T^{k-1}y) \end{bmatrix}.
\]
The above notation will be used throughout the paper.  

\begin{rem}\label{rem:cond_for_preval}
As explained in Subsection~\ref{subsec:preval}, a sufficient condition for a set $\Sk \subset \Lip(X)$ to be prevalent (with the probe set $\{h_1, \ldots, h_m\}$ defined as the family of all real monomials of $N$ variables of degree at most $d$) is that for every Lipschitz observable $h \colon X \to \R$, we have $h_\alpha \in \Sk$ for Lebesgue-almost every $\alpha \in \R^m$. Within the subsequent part of the paper, we check prevalence using this condition. For all the results, it is enough to take $d = 2k+1$ in the definition of $\{h_1, \ldots, h_m\}$, while most of them hold also for $d = 2k-1$, which is indicated in the formulations of the particular results.\footnote{ It is easy to see that if a set $\Sk$ is prevalent with the probe set defined as the family of all monomials of degree at most $d$, then $\Sk$ is prevalent with the probe set defined as the family of all  monomials of degree at most $\tilde d$, for any $\tilde d \ge d$.}
\end{rem}

The following fact was proved in \cite{SauerYorke97}.

\begin{lem}[{\cite[Lemma 4.1]{SauerYorke97}}]\label{lem:SY97}
Fix $d \geq 2k - 1$ and let $\{h_1, \ldots , h_m\}$ be the set of all monomials of $N$ variables of degree at most $d$. Assume $y_1,\ldots, y_{2k}\in \R^N$ satisfy
\begin{align*}
\| y_{i+k} - y_i\| &\geq \sigma \quad \text{for } i = 1, \ldots, k,\\
\|y_i - y_j  \|  &\geq \eps \quad \text{for } i,j = 1, \ldots, 2k \text{ such that }i \neq j,\: |i-j| \neq k.
\end{align*}
for some $\sigma,\eps > 0$. Then for every $z = (z_1, \ldots, z_{2k}) \in \R^{2k}$ there exists $\alpha = (\alpha_1, \ldots, \alpha_m) \in \R^m$, such that 
$$\sum_{j=1}^m \alpha_j h_j(y_i) = z_i \quad \text{for } i = 1, \ldots, 2k$$
and
$$\|\alpha\|_{\infty} \leq \frac{2k (\max \{1,\|y_1\|,\ldots, \|y_{2k}\|\})^{2m-1}\|z\|_{\infty}}{\eps^{2k-2}\sigma}.$$
   
\end{lem}

Using this lemma, we prove the following two estimates of the singular values of $D_{x,y}$.

\begin{prop}\label{prop:singular value on orbit} Fix $d \geq 2k - 1$ and let $\{h_1, \ldots , h_m\}$ be the set of all monomials of $N$ variables of degree at most $d$. For $x , y\in X$ assume
\[
\|T^i \xi_1 - T^j \xi_2 \| \ge \eps \quad \text{for } i,j = 0, \ldots, k-1 \text{ such that }i \neq j,\: \xi_1,\xi_2 \in \{x,y\}
\]
for some $\eps \ge 0$. Then
	\[ \sigma_k(D_{x,y}) \geq C\eps^{2k - 2}\|T^{k-1} x - T^{k-1} y \|, \]
where $C>0$ depends only on $k,m, X$ and $\Lip(T)$.
\end{prop}
\begin{proof}
	Obviously, we can assume $\eps > 0$ and $\|T^{k-1} x - T^{k-1} y \| > 0$.
	Then
	\[   
	\sigma = \min \{ \| T^i x - T^i y \| : 0 \leq i \leq k-1\} > 0.
	\]
	Applying Lemma~\ref{lem:SY97} for $y_i = T^{i-1} x$, $y_{k + i} = T^{i-1} y$, $i = 1, 2, \ldots, k$ and  $z_1, \ldots, z_{k} \in \R$, $z_{k+1} = \cdots = z_{2k} = 0$, we find $\alpha  = (\alpha_1, \ldots, \alpha_m)\in \R^m$ such that
	\[ \sum_{j=1}^m \alpha_j h_j(y_i) = z_i \quad \text{for } i = 1, \ldots, 2k, \qquad \|\alpha\|_{\infty} \leq \frac{2k (\max\{1, \diam X + \dist(0,X)\})^{2m-1}\|\tilde z \|_{\infty}}{\eps^{2k-2}\sigma}, \]
	where
	\[   
	 \tilde{z} = (z_1, \ldots, z_k), \qquad
	 \]
	Hence, by \eqref{eq:D_xy}, 
	\[
	D_{x,y} \alpha = \Big(\sum_{j=1}^m \alpha_j (h_j(T^{i-1} x) - h_j(T^{i-1} y))\Big)_{i= 1}^k =(z_1-z_{k+1},\ldots, z_k-z_{2k}) = \tilde z.
	\]
Moreover,
	 \begin{equation}\label{eq:alpha<}
	  \|\alpha\| \leq \frac{ \|\tilde{z}\|}{c\eps^{2k-2}\sigma}
	  \end{equation}
  	 for some $c > 0$ depending only on $k$, $m$ and $\diam X + \dist(0,X)$. Concluding, for every $\tilde z \in \R^k$ we have found $\alpha \in \R^m$ such that $D_{x,y}\alpha = \tilde z$ and \eqref{eq:alpha<} holds. This implies (it is enough to use the singular value decomposition of $D_{x,y}$, see Subsection~\ref{subsec:sing_values}) that 
  	 \[ \sigma_k(D_{x,y}) \geq c\eps^{2k - 2}\sigma.\]
	 As $T$ is Lipschitz, for $0 \leq i \leq k-1$ we have $\|T^{k-1}x - T^{k-1}y\| \le (\Lip(T))^{k - 1 - i} \|T^i x - T^i y\| \le L^{k - 1} \|T^i x - T^i y\|$, where $L = \max\{\Lip(T),1\}$. Hence, $\sigma \geq L^{1-k}\|T^{k-1} x - T^{k-1} y\|$, which implies
	  \[ \sigma_k(D_{x,y}) \geq cL^{1-k}\eps^{2k - 2}\|T^{k-1} x - T^{k-1} y\|.\]
\end{proof}

\begin{prop}\label{prop:singular value on orbit kernel} 
Fix $d \geq 2k + 1$ and let $\{h_1, \ldots , h_m\}$ be the set of all monomials of $N$ variables of degree at most $d$. For $x , y\in X$  assume
\[
\|T^i \xi_1 - T^j \xi_2 \| \ge \eps \quad \text{for } i,j = 0, \ldots, k \text{ such that }i \neq j,\: \xi_1,\xi_2 \in \{x,y\}
\]
for some $\eps \ge 0$. Then
	\[ \sigma_1(D_{Tx,Ty}|_{\Ker D_{x,y}}) \geq C\eps^{2k}\|T^{k} x - T^{k} y \|, \]
	where $C > 0$ depends only on $k,m, X$ and $\Lip(T)$, while $D_{Tx,Ty}|_{\Ker D_{x,y}}$ denotes the restriction of the linear operator with the matrix $D_{Tx, Ty}$ to the kernel of the linear operator with the matrix $D_{x,y}$.
\end{prop}

\begin{proof}
	The proof is analogous to the one of Proposition~\ref{prop:singular value on orbit}. We can assume $\eps > 0$ and $\|T^{k} x - T^{k} y \| > 0$, which gives
	\[   
	\sigma = \min \{ \| T^i x - T^i y \| : 0 \leq i \leq k\} > 0.
	\]
By Lemma~\ref{lem:SY97} applied for $y_i = T^{i-1} x$, $y_{k + 1 + i} = T^{i-1} y$, $i = 1 \ldots, k + 1$ and 
	\[ z_1 = z_2 = \cdots = z_k =0, \qquad z_{k+1} = w, \qquad z_{k+2} = \cdots = z_{2k+2} = 0,\]
	where $w \in \R$, we find $\alpha = (\alpha_1, \ldots, \alpha_m) \in \R^m$ such that
	\[ \sum_{j=1}^m \alpha_j h_j(y_i) = z_i \quad \text{for } i = 1, \ldots,  2k+2,\qquad  \|\alpha\| \leq \frac{\|\tilde z \|}{c\eps^{2k}\sigma},\]
	where
	\[   
	\tilde{z}=(0,\ldots, 0, w)\in \R^k
	\]
and $c > 0$ depends only on $k$, $m$ and $\diam X + \dist(0,X)$. By \eqref{eq:D_xy}, 
	\begin{align*}   
	D_{x,y} \alpha &= (z_1-z_{k+2},\ldots, z_k-z_{2k+1}) = 0,\\
	D_{Tx,Ty} \alpha &= (z_2-z_{k+3},\ldots, z_{k+1}-z_{2k+2}) = \tilde z.
	\end{align*}
Choosing $w = c\eps^{2k}\sigma$, we have $\alpha \in \Ker D_{x,y} \cap \overline{B_m}(0,1)$ with $\|D_{Tx, Ty}\alpha\|= c\eps^{2k}\sigma$, which implies
\[
\sigma_1(D_{Tx,Ty}|_{\Ker D_{x,y}}) \geq c\eps^{2k}\sigma.
\]
As in the proof of Proposition~\ref{prop:singular value on orbit}, the Lipschitz condition for $T$ gives $\sigma \ge L^{-k}\|T^{k} x - T^{k} y \|$ for $L = \max\{\Lip(T), 1\}$, so
\[
\sigma_1(D_{Tx,Ty}|_{\Ker D_{x,y}}) \geq cL^{-k}\eps^{2k}\|T^{k} x - T^{k} y\|.
\]

\end{proof}

\subsection{Phase space decomposition}\label{subsec:decomp}

The proofs of Theorems~\ref{thm:injectivity sections}--\ref{thm:absolute continuity} require working with energy integrals, which leads to estimates of the correlation dimensions of the considered measures. In order to obtain the results for the Hausdorff dimension, we have to restrict measure $\mu$ to suitable sets, where its Hausdorff and correlation dimensions are arbitrary close. Moreover, for technical reasons it is necessary to consider subsets of $X$, on which the number $\eps$ from Propositions~\ref{prop:singular value on orbit}--\ref{prop:singular value on orbit kernel} is uniformly bounded away from zero. Both objectives are achieved by the following decomposition of $X$.

\begin{prop}\label{prop:decomp} Let $X \subset\R^N$, $N \in \N$, be a compact set, let $\mu$ be a Borel probability measure on $X$ and let $T\colon X \to X$ be a Lipschitz map. Fix $\ell \in \N \cup \{0\}$, $\eta > 0$ and $($in the case $\ell > 0)$ assume $\mu \left( \Prep_\ell(T) \right) = 0$. Then there exists a countable collection $\mF$ of compact subsets of $X$ such that:
	\begin{enumerate}[$($i$)$]
		\item\label{it:F sum} $\mu \left(\bigcup_{F \in \mF} F\right) = 1$ and $\mu(F)>0$ for every $F \in \mF$,
		\item\label{it:cdim>hdim} $\cdim T^{\ell}(\mu|_F) \geq \lhdim T^{\ell}\mu - \eta$,
		\item\label{it:orbit separation} if $\ell > 0$, then for every $F \in \mF$ there exists $\eps = \eps(F)> 0$ such that
		\[
\|T^i \xi_1 - T^j \xi_2 \| \ge \eps \quad \text{for } i,j = 0, \ldots, \ell \text{ such that }i \neq j,\: \xi_1,\xi_2 \in \{x,y\}
\]
for every $x, y \in F$.
	\end{enumerate}
\end{prop}

\begin{proof}
Let $t = \lhdim T^\ell \mu$. For $n \in \N$ set
\[ X_n = \{ x \in X :  T^\ell\mu(B(x,r)) \leq r^{t - \eta} \text{ for every } 0< r< 1/n \}.\]
By the definition of $\lhdim \mu$, we have $\mu \left( \bigcup_{n = 1}^\infty T^{-\ell}(X_{n}) \right) = T^\ell\mu \left( \bigcup_{n = 1}^\infty X_{n} \right) = 1$. Furthermore, $$\cdim T^\ell(\mu|_{T^{-\ell}(X_n)}) \geq t - \eta$$ for every $n \in \N$. Indeed, for $0 < s < t - \eta$, a suitable change of coordinates (see e.g. the calculation in \cite[Chapter~8]{mattila}) and the  equality $T^\ell (\mu|_{T^{-\ell} (X_n)}) = (T^\ell \mu)|_{X_n}$ provide
	\begin{equation*}
		\begin{split}
			\mE_s(T^\ell(\mu|_{T^{-\ell}X_n})) & = \int_{X_n} \int_{X_n} \frac{d(T^\ell \mu) (x) d(T^\ell \mu) (y)}{\| x - y\|^s}  = s  \int_{X_n} \int_{0}^\infty  \frac{T^\ell \mu (X_n \cap B(x,r))}{r^{s+1}} dr d(T^\ell \mu)(x) \\
			& \leq s \int_{0}^{1/n} r^{t - \eta - s - 1} dr + \int_{1/n}^\infty  \frac{dr}{r^{s+1}} < \infty.
		\end{split}
	\end{equation*}
The remainder of the construction proceeds as in the proof of \cite[Theorem~4.2]{SauerYorke97}. We present the arguments for the reader's convenience. In the case $\ell > 0$ we have $\mu ( \Prep_\ell(T)) = 0$, so
\[ \eps(x) = \min\{ \|T^i x - T^j x \| : 0 \leq i \neq j \leq \ell \} > 0 \quad \text{ for } \mu\text{-almost every } x \in X\]
(in the case $\ell = 0$ we set $\eps(x) = 1$ for $x \in X$). Then for
\[ Y_q = \{ x \in X : \eps(x) \geq 1/q \}, \qquad q \in \N \]
we obtain $\mu ( \bigcup_{q = 1}^\infty Y_{q} ) = 1$. Note that if $y \in \overline{B_N}(x, \frac{1}{4q}(\Lip(T))^{-\ell})$ for some $x \in Y_q$, then 
\[
\|T^i \xi_1 - T^j \xi_2 \| \ge \frac{1}{2q} \quad \text{for } i,j = 0, \ldots, \ell \text{ such that }i \neq j,\: \xi_1,\xi_2 \in \{x,y\}.
\]
Let $\mB_q$ be a countable cover of $Y_{q}$ by balls centred in $Y_q$, of radii at most $\frac{1}{4q}(\Lip(T))^{-\ell}$. Then
\[ \tilde\mF = \{ \tilde F = T^{-\ell}X_{n} \cap Y_{q} \cap B : \: n, q \in \N,\ B \in \mB_q,\ \mu(\tilde F) > 0 \} \]
is a countable family satisfying the conditions \ref{it:F sum}--\ref{it:orbit separation}. By the regularity of $\mu$, each set $\tilde F \in \tilde \mF$ has a full $\mu$-measure subset, which is a countable union of compact sets. The union over $\tilde F \in \tilde \mF$ of the families of all these compact sets defines the suitable family $\mF$.

\end{proof}

\subsection{Restricted conditional measures}\label{subsec:restricted cond meas}

In the subsequent proofs, we will use the following fact.

\begin{lem}\label{lem:restricted cond meas} Let $X \subset \R^N$, $N \in \N$, be a compact set, let $\mu$ be a Borel probability measure on $X$ and let $\phi\colon X \to \R^k$, $k \in \N$, be a Borel map. Suppose $F \subset X$ is a positive $\mu$-measure set and let $\nu = \frac{1}{\mu(F)}\mu|_F$. Then for $\nu$-almost every $x \in F$, there exists $f(x) > 0$ such that
\[
\nu_{\phi, \phi(x)} = f(x) \mu_{\phi, \phi(x)}|_F,
\]
where $\{\mu_{\phi, y}\}_{y\in \R^k}$ and $\{\nu_{\phi, y}\}_{y\in \R^k}$ are, respectively, the systems of conditional measure of $\mu$ and $\nu$, with respect to $\phi$.
\end{lem}
\begin{proof} Note that $\phi(\mu|_F) \ll \phi\mu$, so by the differentiation theorem for measures (see e.g.~\cite[Theorem~2.12]{mattila}), the Radon-Nikodym derivative 
			\[\frac{d\phi(\mu|_F)}{d\phi \mu}(y) = \lim_{\delta \to 0} \frac{\phi(\mu|_F)(B(y,\delta))}{\phi\mu(B(y,\delta))}  \]
			exists and is positive and finite for $\phi(\mu|_F)$-almost every $y \in \R^k$. This together with \eqref{eq:cond measure limit} implies the assertion of the lemma for $f(x) = \left(\mu(F)\frac{d\phi(\mu|_F)}{d\phi \mu}(\phi(x)) \right)^{-1}$.
\end{proof}

\subsection{Geometric slices}\label{subsec:slices}

In the subsequent section, apart from the conditional measures defined in \eqref{eq:cond measure limit}, we use `geometric slices' of the measure $\mu$. 
More precisely, for a Borel map $\phi \colon X \to \R^k$ on a compact set $X \subset \R^N$ and a Borel probability measure $\mu$ on $X$ we define the system of \emph{geometric slices} $\mu^G_{\phi,y}$, $y \in \R^k$, of $\mu$ as weak-$^*$ limits (recall that $\mH^s$, $s > 0$ denotes the $s$-dimensional Hausdorff measure)
\[ \mu^G_{\phi,y} = \lim_{\delta \to 0}\ \frac{1}{\mH^k(B(y, \delta))} \mu|_{\phi^{-1}(B(y,\delta))} = \frac{1}{\mH^k(B(0,1))}\lim_{\delta \to 0}\frac{1}{\delta^k} \mu|_{\phi^{-1}(B(y,\delta))},\]
whenever the limit exists, and zero otherwise. The relation between the conditional measures and geometric slices (under an absolute continuity condition) is explained by the following lemma.

\begin{lem}\label{lem: geometric to conditional}
Let $\phi \colon X \to \R^k$, $k \in \N$, be a Borel map on a compact set $X \subset \R^N$, $N \in \N$, and let $\mu$ be a probability Borel measure on $X$. Assume $\phi \mu \ll \mH^k$. Then for $\phi \mu$-almost every $y \in \R^k$ there exists $0 < f(y) < \infty$ such that
\begin{equation*}\label{eq:geometric to conditional} \mu^G_{\phi,y} = f(y) \mu_{\phi,y}.
\end{equation*}
Moreover, 
\[ \mu(E) = \int_{\R^k} \mu^G_{\phi,y}(E)d\mH^k(y). \]
for every Borel set $E \subset X$.
\end{lem}
\begin{proof}
Set $f$ to be the Radon--Nikodym derivative $\frac{d(\phi \mu)}{d\mH^k}$. Then $0 < f(y) < \infty$ for $\phi \mu$-almost every $y \in \R^k$ and  $f(y) = \lim_{\delta \to 0} \frac{\phi \mu(B(y,\delta))}{\mH^k(B(y,\delta))}$ for $\mH^k$-almost every $y \in \R^k$, which provides the first assertion of the lemma. Furthermore, for every Borel set $E \subset X$,
\[ \mu(E) = \int_{\R^k} \mu_{\phi,y}(E)d(\phi \mu)(y) = \int_{\R^k} \mu_{\phi,y}(E)f(y)d\mH^k(y) = \int_{\R^k} \mu^G_{\phi,y}(E)d\mH^k(y).  \] 
\end{proof}

We will also make use of the following simple observation. If $g \colon X \to [0, \infty]$ is lower semi-continuous, then for $\phi \mu$-almost every $y \in \R^k$,
\begin{equation}\label{eq:muG liminf}
	\int g\, d \mu^G_{\phi,y} \leq  \frac{1}{\mH^k(B(0,1))}\liminf_{\delta \to 0} \frac{1}{\delta^k} \int_{\phi^{-1}(B(y,\delta))} g\, d\mu.
\end{equation}
This follows from the definition of $\mu^G_{\phi,y}$ as a weak-$^*$ limit and the fact that a lower semi-continuous function $g \colon X \to [0, \infty]$ is a non-decreasing limit of a sequence of non-negative continuous functions.

\begin{rem}\label{rem:cond_vs_slices} The advantage of switching from conditional measures to geometric slices is that the latter are better suited for the transversality argument used in the proofs. In Theorem~\ref{thm:absolute continuity full} we will show that under assumptions of our main theorems, the condition $\phi \mu \ll \mH^k$ is satisfied, when $\phi$ is a (typical) $k$-delay coordinate map. Then by Lemma~\ref{lem: geometric to conditional}, the dimensions of $\mu_{\phi,y}$ and $\mu^G_{\phi,y}$ coincide for $\phi \mu$-almost every $y \in \R^k$. 
\end{rem}

\section{Proofs of Theorems~\ref{thm:injectivity sections}--\ref{thm:absolute continuity}}\label{sec:proof sections}

In this section we prove suitable versions of Theorems~\ref{thm:injectivity sections}--\ref{thm:absolute continuity}, using the prevalence condition described in Remark~\ref{rem:cond_for_preval}, under the notation introduced in Subsection~\ref{subsec:matrices}.
\subsection{Proof of Theorem~\ref{thm:absolute continuity}} \label{subsec:proof abs cont}

The following result is the suitable version of Theorem~\ref{thm:absolute continuity}.

\begin{thm}\label{thm:absolute continuity full}
Let $X \subset\R^N$, $N \in \N$, be a compact set, let $\mu$ be a Borel probability measure on $X$ and let $T\colon X \to X$ be a Lipschitz map. Fix $k \in \N$ such that $k < \lhdim T^{k-1} \mu$ and $($in the case $k > 1)$ assume $\mu ( \Prep_{k-1}(T) ) = 0$. Fix $d \geq 2k - 1$ and let $\{h_1, \ldots , h_m\}$ be the set of all monomials of $N$ variables of degree at most $d$. Let $h \colon X \to \R$ be a Lipschitz observable. Then for Lebesgue-almost every $\alpha \in \R^m$ we have $\phi_{\alpha} \mu \ll \mH^k$, where $\phi_{\alpha}$ is the $k$-delay coordinate map corresponding to $h_\alpha$.

If, additionally, $\mu$ is $T$-invariant or $T$ is bi-Lipschitz onto its image, then the condition $k < \lhdim T^{k-1}\mu$ may be replaced by $k < \lhdim \mu$ and the condition $\mu(\Prep_{k-1}(T)) = 0$ may be replaced by $\mu\big(\bigcup_{p=1}^{k-1} \Per_p(T)\big) = 0$.

\end{thm}

\begin{proof} It is easy to check that it is sufficient to prove the assertion for  Lebesgue-almost every $\alpha \in B_m(0,1)$ (then the general case follows by multiplying $h$ by a positive constant). 
Choose $\eta > 0$ such that $\lhdim T^{k-1} \mu - \eta > k$. Consider the collection $\mF$ from Proposition~\ref{prop:decomp} corresponding to $\ell = k-1$ and $\eta$. Fix $F \in \mF$. We will prove that $\phi_\alpha(\mu|_F) \ll \mH^{k}$ for Lebesgue-almost every $\alpha \in B_m(0,1)$. As $\mu \left(\bigcup_{F \in \mF} F\right) = 1$, this will finish the proof.

To prove $\phi_\alpha(\mu|_F) \ll \mH^{k}$, it suffices to show (see \cite[Theorem~2.12]{mattila}) that for $\phi_\alpha(\mu|_F)$-almost every $z \in \R^k$ we have
\[ 
\liminf_{\delta \to 0} \frac{\phi_\alpha(\mu|_F)(B(z,\delta))}{\delta^k} < \infty.
\]
Therefore, it is enough to show
\[ I = \int_{B_m(0,1)} \int_{\R^k}\liminf_{\delta \to 0} \frac{\phi_\alpha(\mu|_F)(B(z,\delta))}{\delta^k} \: d\phi_\alpha(\mu|_F)(z)d\alpha < \infty. \]
For that, we proceed as follows. First, by Fatou's lemma,
\begin{align*} I & \leq \liminf_{\delta \to 0} \frac{1}{\delta^k}  \int_{B_m(0,1)} \int_{F} \mu\left( F \cap \phi^{-1}_\alpha( B(\phi_\alpha(x), \delta))\right)d\mu(x)d\alpha \\
	& =  \liminf_{\delta \to 0} \frac{1}{\delta^k}  \int_{B_m(0,1)} \int_{F} \int_{F} \mathds{1}_{\phi^{-1}_\alpha( B(\phi_\alpha(x), \delta))}(y) \: d\mu(y)d\mu(x)d\alpha.
\end{align*}
Consequently, by Tonelli's theorem, \eqref{eq:matrix_form} and Lemma~\ref{lem: key_ineq_inter},
\begin{align*}
	I & \leq \liminf_{\delta \to 0} \frac{1}{\delta^k}  \int_{F} \int_{F} \int_{B_m(0,1)}  \mathds{1}_{ \{ \alpha \in B_m(0,1) : \| \phi_\alpha(x) - \phi_\alpha(y) \| \leq \delta \}}(\alpha) \: d\alpha d\mu(x)d\mu(y) \\
	& = \liminf_{\delta \to 0} \frac{1}{\delta^k}  \int_{F} \int_{F} \Leb\left( \{ \alpha \in B_m(0,1) : \| \phi_\alpha(x) - \phi_\alpha(y) \| \leq \delta \} \right) d\mu(x)d\mu(y) \\
		& = \liminf_{\delta \to 0} \frac{1}{\delta^k}  \int_{F} \int_{F} \Leb\left( \{ \alpha \in B_m(0,1) : \|D_{x,y} \alpha + w_{x,y} \| \leq \delta \} \right) d\mu(x)d\mu(y) \\
		& \leq C_1 \int_{F} \int_{F} \frac{d\mu(x)d\mu(y)}{(\sigma_k(D_{x,y}))^k}
	\end{align*}
for some $C_1 > 0$. By Proposition~\ref{prop:decomp}.\ref{it:orbit separation}, we can apply Proposition~\ref{prop:singular value on orbit} to obtain
\[
	I \leq C_2 \int_{F} \int_{F} \frac{d\mu(x)d\mu(y)}{\|T^{k-1} x - T^{k-1} y\|^k}  = C_2  \mE_k(T^{k-1} (\mu|_F)) 
\]
for some $C_2 = C_2(F) > 0$. Since $\cdim T^{k-1} (\mu|_F) \geq \lhdim T^{k-1} \mu - \eta  > k$ by Proposition~\ref{prop:decomp}\ref{it:cdim>hdim}, we obtain $\mE_k(T^{k-1} (\mu|_F)) < \infty$, which gives $I < \infty$ and ends the proof of the absolute continuity of $\phi_\alpha \mu$.

To show the additional assertions, note that a push-forward by a bi-Lipschitz map does not change the dimension of the measure and injectivity of $T$ implies that every pre-periodic point $x \in \Prep_{k-1}$ is actually periodic of period at most $k-1$. Furthermore, if $\mu$ is $T$-invariant and $\mu\big(\bigcup_{p=1}^{k-1} \Per_p(T)\big) = 0$, then $\mu(\Prep_{k-1}(T)) = 0$ since  $\Prep_{k-1}(T) \subset T^{-(k-1)}(\bigcup_{p=1}^{k-1}\Per_p(T))$.
\end{proof}

\subsection{Proof of Theorem~\ref{thm:injectivity sections}} \label{subsec:proof injectivity sections}
To prove Theorem~\ref{thm:injectivity sections}, first we show the following result, where we assume $\mu ( \Prep_{k-1}(T)) = 0$. Later we will explain how to remove this assumption. For simplicity, here and in the sequel we write $\mu_{\alpha, y}$ (resp.~$\mu^G_{\alpha, y}$), $y \in \R^k$, for conditional measures (resp.~geometric slices) of a Borel probability measure $\mu$ on $X$ with respect to a $k$-delay coordinate map $\phi_\alpha$.

\begin{thm}\label{thm:injectivity sections lower full}
Let $X \subset\R^N$, $N \in \N$, be a compact set, let $\mu$ be a Borel probability measure on $X$ and let $T\colon X \to X$ be a Lipschitz map. Fix $k \in \N$ and $($in the case $k > 1)$ assume $\mu ( \Prep_{k-1}(T) ) = 0$. Fix $d \geq 2k - 1$ and let $\{h_1, \ldots , h_m\}$ be the set of all monomials of $N$ variables of degree at most $d$. Let $h \colon X \to \R$ be a Lipschitz observable. Then for Lebesgue-almost every $\alpha \in \R^m$,
\[ \lhdim T^{k-1} (\mu_{\alpha, \phi_\alpha(x)}) \geq \lhdim T^{k-1} \mu - k \quad \text{for } \mu\text{-almost every } x \in X, \]
where $\phi_{\alpha}$ is the $k$-delay coordinate map corresponding to $h_\alpha$.
\end{thm}

\begin{proof} As previously, it is sufficient to prove the assertion for Lebesgue-almost every $\alpha \in B_m(0,1)$. 
	Obviously, we can assume $\lhdim T^{k-1} \mu > k$. Choose $\eta,s > 0$ such that 
	\[   
\eta < \lhdim T^{k-1} \mu -k, \qquad k < s < \lhdim T^{k-1} \mu - \eta.
	\]
  Consider the collection $\mF$ from Proposition~\ref{prop:decomp} corresponding to $\ell = k-1$ and $\eta$. Fix $F \in \mF$. We will prove
	\begin{equation}\label{eq:aux}
	 \lhdim  T^{k-1} (\mu_{\alpha, \phi(x)}) \geq s - k\text{ for } \mu\text{-almost every } x \in F \text{ and almost every } \alpha \in B_m(0,1).
	\end{equation}
	As $\eta, s$ can be chosen such that $s$ is arbitrarily close to $\lhdim T^{k-1} \mu$ and $\mu \left(\bigcup_{F \in \mF} F\right) = 1$, proving \eqref{eq:aux} will conclude the proof of the theorem.
	
	Let
	\[\nu = \frac{1}{\mu(F)}\mu|_F.\] 
	As by Lemma~\ref{lem:restricted cond meas}, for $\nu$-almost every $x \in F$, the measures $\nu_{\alpha, \phi_\alpha(x)}$ and $\mu_{\alpha, \phi_\alpha(x)}|_F$ are equal up to a multiplication by a positive constant, to prove \eqref{eq:aux} it is sufficient to show
	\begin{equation}\label{eq: nu lhdim bound injectivity} \lhdim T^{k-1} (\nu_{\alpha,\phi_\alpha(x)}) \geq s - k \quad \text{for } \nu\text{-almost every } x \in F \text{ and $\Leb$-almost every } \alpha \in B_m(0,1).
	\end{equation}
	The remainder of the proof is devoted to verifying \eqref{eq: nu lhdim bound injectivity}. 
	
	As $\lhdim T^{k-1} \nu \geq \lhdim T^{k-1} \mu > k$ and $($in the case $k > 1)$ $$\nu ( \Prep_{k-1}(T) ) \le \frac{1}{\mu(F)}\mu ( \Prep_{k-1}(T) ) = 0,$$ Theorem~\ref{thm:absolute continuity full} implies $\phi_\alpha \nu \ll \mH^k$ for Lebesgue-almost every $\alpha \in B_m(0,1)$. Hence, by Lemma~\ref{lem: geometric to conditional}, the assertion \eqref{eq: nu lhdim bound injectivity} will follow from
	\begin{equation}\label{eq: nu lhdim bound injectivity2}\lhdim   T^{k-1} (\nu^G_{\alpha, z}) \geq s - k \quad \text{for } \mH^k\text{-almost every } z \in \R^k \text{ and $\Leb$-almost every } \alpha \in B_m(0,1)
	\end{equation}
	(with the convention that $\lhdim$ of a zero measure is zero). By Lemma~\ref{lem:corr to Hausdorff}, to show \eqref{eq: nu lhdim bound injectivity2}  it suffices to prove
	\begin{equation}\label{eq:inject slices main integral}
		I = \int_{B_m(0,1)} \int_{\R^k} \mE_{s-k}(T^{k-1} (\nu^G_{\alpha,z})) d\mH^k(z)d\alpha < \infty.
	\end{equation}
	
	Now we check \eqref{eq:inject slices main integral}. As the function $x \mapsto \|T^{k-1}x - T^{k-1}y\|^{k-s} \in [0, +\infty]$ is lower semi-continuous, applying \eqref{eq:muG liminf} for the measure $\nu$ together with Fatou's lemma implies
	\begin{align*}
		I & = \int_{B_m(0,1)} \int_{\R^k} \int_F  \int_F   \frac{d\nu^G_{\alpha, z}(x) d\nu^G_{\alpha, z}(y) d\mH^k(z)d\alpha}{\|T^{k-1}x - T^{k-1}y\|^{s-k}} \\
		& \leq \liminf_{\delta \to 0} \frac{C_1}{\delta^k} \int_{B_m(0,1)} \int_{\R^k} \int_F  \int_{\phi_\alpha^{-1}(B(z, \delta))}   \frac{d\nu(x)d\nu^G_{\alpha, z}(y) d\mH^k(z)d\alpha}{\|T^{k-1}x - T^{k-1}y\|^{s-k}}
	\end{align*}
	for some $C_1>0$. By Lemma~\ref{lem: geometric to conditional},
	\[ \int_{\R^k} \int_{F} \frac{d\nu^G_{\alpha,z}(y)d\mH^k(z)}{\|T^{k-1}x - T^{k-1}y\|^{s-k}} = \int_{F}  \frac{d\nu(y)}{\|T^{k-1}x - T^{k-1}y\|^{s-k}}. \]
	Therefore, by Tonelli's theorem,
	\begin{align*} 	I & \leq \liminf_{\delta \to 0} \frac{C_1}{\delta^k} \int_{B_m(0,1)} \int\limits_{F}  \int_{\phi_\alpha^{-1}(B(\phi_\alpha(y), \delta))}   \frac{d\nu(x) d \nu(y) d\alpha}{\|T^{k-1}x - T^{k-1}y\|^{s-k}} \\
		& =  \liminf_{\delta \to 0} \frac{C_1}{\delta^k} \int_{B_m(0,1)} \int_{F}  \int_F   \frac{\mathds{1}_{\phi_\alpha^{-1}(B(\phi_\alpha(y), \delta))}(x)}{\|T^{k-1}x -T^{k-1}y\|^{s-k}} \: d\nu(x) d \nu(y) d\alpha\\
		& = \liminf_{\delta \to 0} \frac{C_1}{\delta^k}  \int_{F}  \int_F \int_{B_m(0,1)} \frac{\mathds{1}_{\{ \alpha \in B_m(0,1) : \|\phi_\alpha(x) - \phi_\alpha(y)\| \leq \delta \}}(\alpha)}{\|T^{k-1}x - T^{k-1}y\|^{s-k}} \: d\alpha d\nu(x) d \nu(y) \\
		& = \liminf_{\delta \to 0} \frac{C_1}{\delta^k}  \int_{F}  \int_F \frac{\Leb(\{ \alpha \in B_m(0,1) : \|\phi_\alpha(x) - \phi_\alpha(y)\| \leq \delta \})}{\|T^{k-1}x - T^{k-1}y\|^{s-k}} \: d\nu(x) d \nu(y).
	\end{align*}
	Consequently, by Lemma~\ref{lem: key_ineq_inter},
	\[	I \leq C_1\int_{F}  \int_F  \frac{d\nu(x) d \nu(y)}{(\sigma_k(D_{x,y}))^k \|T^{k-1}x - T^{k-1}y\|^{s-k}}.\]
	By Proposition~\ref{prop:decomp}\ref{it:orbit separation}, we can apply Proposition~\ref{prop:singular value on orbit} to obtain
	\[ I \leq C_2 \int_{F}  \int_F  \frac{d\nu(x) d \nu(y)}{\|T^{k-1}x - T^{k-1}\|^s} = C_2\mE_s(T^{k-1} \nu).\]
	for some $C_2 = C_2(F)> 0$. As $\cdim T^k \nu \geq \lhdim T^k \mu - \eta > s$ by Proposition~\ref{prop:decomp}.\ref{it:cdim>hdim}, we have $\mE_s(T^{k-1} \nu) < \infty$, which establishes \eqref{eq:inject slices main integral} and finishes the proof of the theorem.

	\end{proof}

\begin{proof}[Proof of Theorem~\rm\ref{thm:injectivity sections}]
	Assertion~\ref{it:injectivity sections lhdim} of Theorem~\ref{thm:injectivity sections} is a direct consequence of Theorem~\ref{thm:injectivity sections lower full}, if we assume additionally $\mu \left( \Prep_{k-1}(T) \right) = 0$ in the case $k>1$. Therefore, it remains to prove that this assumption can be omitted.
	
	For $p \in \N \cup \{0\}$ and $q \in \N$ define
	\[X_{p,q} = \{ x \in X : x, Tx, \ldots, T^{p + q -1}x \text{ are pairwise distinct and } T^{p+q}x = T^{p} x \}.\]
	In other words, $X_{p,q}$ is the union of all pre-periodic orbits under $T$ with the pre-periodic part of length $p$ and the periodic part of length $q$. Let
	\[ \mR = \{ (p,q) : p \in \N \cup \{0\},\, q \in \N,\, p+q \leq k-1 \}, \qquad Y = X \setminus \bigcup_{(p,q) \in \mR} X_{p,q}\]
	and set
	\[ 
	\mu^{p,q} = 
	\begin{cases}   
	\frac{1}{\mu(X_{p,q})}\mu|_{X_{p,q}} &\text{if } \mu(X_{p,q}) > 0\\
	0 &\text{if } \mu(X_{p,q}) = 0
	\end{cases}
, \qquad \nu = 
\begin{cases}   
	\frac{1}{\mu(Y)}\mu|_Y &\text{if } \mu(Y) > 0\\
	0 &\text{if } \mu(Y) = 0
	\end{cases}. \]
	As
	\[ \mu = \mu|_Y + \sum_{(p,q) \in \mR} \mu|_{X_{p,q}},\]
	we have that for $\mu$-almost every $x \in X$, the conditional measure $\mu_{\phi_{h,k}, \phi_{h,k}(x)}$ with respect to a $k$-delay coordinate map $\phi_{h,k}$ corresponding to an observable $h$ (here and in the sequel it is convenient to include the number of measurements in the notation for the delay coordinate map) is a convex combination of the conditional measures $\mu^{p,q}_{\phi_{h,k}, \phi_{h,k}(x)}$, $(p,q) \in \mR$, and $\nu_{\phi_{h,k}, \phi_{h,k}(x)}$. Therefore, it suffices to show for a prevalent Lipschitz observable $h\colon X \to \R$, 
	\begin{alignat*}{3}   
		\lhdim T^{k-1} (\mu^{p,q}_{\phi_{h,k}, \phi_{h,k}(x)}) &\geq \lhdim T^{k-1} \mu - k \qquad &&\text{for } \mu^{p,q}\text{-almost every } x \in X,\\
	\lhdim T^{k-1} (\nu_{\phi_{h,k}, \phi_{h,k}(x)}) &\geq \lhdim T^{k-1} \mu - k \qquad &&\text{for } \nu\text{-almost every } x \in X
	\end{alignat*}
	for $(p,q) \in \mR$. As
	\[\lhdim T^{k-1} \mu^{p,q} \geq \lhdim T^{k-1} \mu, \qquad  \lhdim T^{k-1} \nu \geq \lhdim T^{k-1} \mu\]
	whenever these measure are non-zero, it suffices to prove, for a prevalent $h$ and $(p,q) \in \mR$.
	\begin{alignat}{3}\label{eq:mu pq lhdim lower}
		\lhdim T^{k-1} (\mu^{p,q}_{\phi_{h,k}, \phi_{h,k}(x)}) &\geq \lhdim T^{k-1} \mu^{p,q} - k &\quad &\text{for } \mu^{p,q}\text{-almost every } x \in X,\\
	\label{eq:nu lhdim lower}
		\lhdim T^{k-1} (\nu_{\phi_{h,k}, \phi_{h,k}(x)}) &\geq \lhdim T^{k-1} \nu - k &\quad &\text{for } \nu\text{-almost every } x \in X.
	\end{alignat}
	As $\nu(\Prep_{k-1}(T)) = 0$, Theorem~\ref{thm:injectivity sections lower full} gives \eqref{eq:nu lhdim lower}. Fix $(p,q) \in \mR$ and let $\ell = p + q \leq k-1$. Then $\mu^{p,q}(\Prep_{\ell - 1}(T)) = 0$, so we can apply Theorem~\ref{thm:injectivity sections lower full} to the measure $\mu^{p,q}$, with $\ell$ instead of $k$. This implies
	\begin{equation}\label{eq:mu pq lhdim lower ell}
		\lhdim T^{\ell-1} (\mu^{p,q}_{\phi_{h,\ell}, \phi_{h,\ell}(x)}) \geq \lhdim T^{\ell-1} \mu^{p,q} - \ell \quad \text{for } \mu^{p,q}\text{-almost every } x \in X
	\end{equation}
	for a prevalent $h$. Since for $x, y \in X_{p,q}$ we have $\|\phi_{h,\ell}(x) - \phi_{h,\ell}(y)\|_\infty = \|\phi_{h,k}(x) - \phi_{h,k}(y)\|_\infty$, we see from \eqref{eq:cond measure limit} that 
	\begin{equation}\label{eq:periodic cond measures equal} \mu^{p,q}_{\phi_{h,\ell},\phi_{h,\ell}(x)} = \mu^{p,q}_{\phi_{h,k}, \phi_{h,k}(x)} \quad \text{for } \mu^{p,q}\text{-almost every }x \in X.
	\end{equation}
	As $T$ is Lipschitz and $\ell < k$, we have
	\begin{equation}\label{eq:mu pq dim iterate}
		\lhdim T^{\ell-1} \mu^{p,q} - \ell \geq \lhdim T^{k-1} \mu^{p,q} - k.
	\end{equation}
	Note that there exists $0 \leq i \leq \ell - 1$ such that $T^{k-1} x = T^{i} x$ for every $x \in X_{p,q}$. Indeed, let $a \in \N \cup \{0\}$ be the maximal number such that $k - 1 - aq \geq p+q$. Then $i =k - 1 - (a+1)q$ satisfies $0 \leq p \leq i \leq p+q-1 = \ell - 1$ and $T^{k-1} x = T^{k - 1 - (a+1)q} x = T^{i} x$ for every $x \in X_{p,q}$. Therefore,
	\[ \lhdim T^{k-1} (\mu^{p,q}_{\phi_{h,k}, \phi_{h,k}(x)}) = \lhdim T^{i} (\mu^{p,q}_{\phi_{h,k},\phi_{h,k}(x)}) \geq \lhdim T^{\ell-1} (\mu^{p,q}_{\phi_{h,k}, \phi_{h,k}(x)}).
	\]
	Combining this with \eqref{eq:mu pq lhdim lower ell}, \eqref{eq:periodic cond measures equal} and \eqref{eq:mu pq dim iterate} yields \eqref{eq:mu pq lhdim lower} and finishes the proof of assertion~\ref{it:injectivity sections lhdim} of Theorem~\ref{thm:injectivity sections}.

	To prove assertion~\ref{it:injectivity sections uhdim}, note that for every $\eps > 0$ there exists a set Borel $E \subset X$ of positive $\mu$-measure such that $\lhdim T^{k-1} (\mu|_E) \geq \hdim T^{k-1} \mu - \eps$ (this can be easily seen from the definitions of the upper and lower Hausdorff dimension of measures via local dimensions, see Subsection~\ref{subsec:dim}). The assertion is then proved by applying the already established point \ref{it:injectivity sections lhdim}  of Theorem \ref{thm:injectivity sections} to the measure $\tilde \mu = \frac{1}{\mu(E)}\mu|_E$ and noting that $T^{k-1} (\tilde \mu_{h, \phi_h(x)})  \ll T^{k-1} (\mu_{h,\phi_h(x)})$, so $\lhdim T^{k-1} (\tilde \mu_{h, \phi_h(x)}) \leq \hdim T^{k-1} (\tilde \mu_{h, \phi_h(x)}) \leq \hdim  T^{k-1} (\mu_{h,\phi_h(x)})$ for $\mu$-almost every $x \in E$.

\end{proof}

\subsection{Proof of Theorem~\ref{thm:predictability sections}}\label{subsec:proof predictability sections}

First, we prove the following.

\begin{thm}\label{thm:predictability sections lower full}
Let $X \subset\R^N$, $N \in \N$, be a compact set, let $\mu$ be a Borel probability measure on $X$ and let $T\colon X \to X$ be a Lipschitz map. Fix $k \in \N$ and assume $\mu (\Prep_k(T)) = 0$. Fix $d \geq 2k + 1$ and let $\{h_1, \ldots , h_m\}$ be the set of all monomials of $N$ variables of degree at most $d$. Let $h \colon X \to \R$ be a Lipschitz observable. Then for Lebesgue-almost every $\alpha \in \R^m$,
	\[ \lhdim \, \phi_\alpha \circ T  (\mu_{\alpha, \phi_\alpha(x)}) \geq \min\{ 1, \lhdim T^{k} \mu - k\} \quad \text{for } \mu\text{-almost every } x \in X,\]
	where $\phi_\alpha$ is the $k$-delay coordinate map corresponding to $h_\alpha$.
\end{thm}

\begin{proof} Again, it is sufficient to prove the assertion for Lebesgue-almost every $\alpha \in B_m(0,1)$. The first part of the proof is similar to the one of Theorem~\ref{thm:injectivity sections lower full}. We can assume $\lhdim T^k \mu > k$. Choose $\eta,s > 0$ such that 
\[
\eta < \lhdim T^k \mu - k, \qquad k < s < k + \min\{1, \lhdim T^k \mu - \eta - k\}.
\]
Consider the collection $\mF$ from Proposition~\ref{prop:decomp} corresponding to $\ell = k$ and $\eta$, and fix $F \in \mF$. We will show
\begin{equation}\label{eq:aux1}
\lhdim \, \phi_h \circ T (\mu_{\alpha, \phi_\alpha(x)} )\geq s - k \quad \text{for } \mu\text{-almost every } x \in F \text{ and almost every } \alpha \in B_m(0,1).
\end{equation}
As $\eta, s$ can be chosen such that $s-k$ is arbitrarily close to $\min\{ 1, \lhdim T^{k} \mu - k\}$ and $\mu \left(\bigcup_{F \in \mF} F\right) = 1$, showing \eqref{eq:aux1} will prove the theorem. 

Let 
\[
\nu = \frac{1}{\mu(F)}\mu|_F.
\]
As by Lemma~\ref{lem:restricted cond meas}, for $\nu$-almost every $x \in F$, the measures $\nu_{\alpha, \phi_\alpha(x)}$ and $\mu_{\alpha, \phi_\alpha(x)}|_F$ are equal up to a multiplication by a positive constant, to show \eqref{eq:aux1} it is sufficient to prove
\begin{equation}\label{eq: nu lhdim bound} \lhdim \, \phi_h \circ T (\nu_{\alpha, \phi_\alpha(x)}) \geq s - k \quad \text{for } \nu\text{-almost every } x \in F \text{ and almost every } \alpha \in B_m(0,1).
\end{equation}
In the remainder of the proof we show \eqref{eq: nu lhdim bound}. As $\lhdim T^{k-1} \nu \geq \lhdim T^{k} \nu \geq \lhdim T^k \mu  > k$ and $\nu(\Prep_k(T)) \leq \frac{1}{\mu(F)}\mu (\Prep_k(T)) = 0$, Theorem~\ref{thm:absolute continuity full} implies $\phi_\alpha \nu \ll \mH^k$ for Lebesgue-almost every $\alpha \in B_m(0,1)$. Consequently, by Lemma~\ref{lem: geometric to conditional}, the assertion \eqref{eq: nu lhdim bound} will follow from
	\begin{equation} \label{eq: nu lhdim bound2}\lhdim \, \phi_h \circ T   (\nu^G_{\alpha,z}) \geq s - k \quad \text{for } \mH^k\text{-almost every } z \in \R^k \text{ and almost every } \alpha \in B_m(0,1).
	\end{equation}
	By Lemma~\ref{lem:corr to Hausdorff}, to show \eqref{eq: nu lhdim bound2}  it is sufficient to prove
	\begin{equation}\label{eq:predict slices main integral}
		I = \int_{B_m(0,1)} \int_{\R^k} \mE_{s-k}(\phi_\alpha \circ T (\nu^G_{\alpha,z}))  d\mH^k(z)d\alpha < \infty.
	\end{equation}
	Now we verify \eqref{eq:predict slices main integral}. As the function $x \mapsto \|\phi_\alpha(Tx) - \phi_\alpha(Ty)\|^{k-s} \in [0, +\infty]$ is lower semi-continuous, \eqref{eq:muG liminf} and Fatou's lemma imply
	\begin{align*}
		I &= \int_{B_m(0,1)} \int_{\R^k} \int_F  \int_F  \frac{d\nu^G_{\alpha, z}(x) d\nu^G_{\alpha, z}(y) d\mH^k(z)d\alpha}{\|\phi_\alpha(Tx) - \phi_\alpha(Ty)\|^{s-k}} \\
		&\leq \liminf_{\delta \to 0} \frac{C_1}{\delta^k} \int_{B_m(0,1)} \int_{\R^k} \int_F  \int_{\phi_\alpha^{-1}(B(z, \delta))}  \frac{d\nu(x)d\nu^G_{\alpha, z}(y) d\mH^k(z)d\alpha}{ \|\phi_\alpha(Tx) - \phi_\alpha(Ty)\|^{s-k}}
	\end{align*}
	for some $C_1>0$. By Lemma~\ref{lem: geometric to conditional},
		\[ \int_{\R^k} \int_{F} \frac{d\nu^G_{\alpha,z}(y)d\mH^k(z)}{ \|\phi_\alpha(Tx) - \phi_\alpha(Ty) \|^{s-k}} = \int_{F} \frac{d\nu(y)}{\|\phi_\alpha(Tx) - \phi_\alpha(Ty) \|^{s-k} },\]
	so by Tonelli's theorem,
	\begin{align*} 	I & \leq \liminf_{\delta \to 0} \frac{C_1}{\delta^k} \int_{B_m(0,1)} \int_{F}  \int_{\phi_\alpha^{-1}(B(\phi_\alpha(y), \delta))}   \frac{d\nu(x) d \nu(y) d\alpha}{\|\phi_\alpha(Tx) - \phi_\alpha(Ty)\|^{s-k}} \\
		& =  \liminf_{\delta \to 0} \frac{C_1}{\delta^k} \int_{B_m(0,1)} \int_{F}  \int_F \frac{\mathds{1}_{\phi_\alpha^{-1}(B(\phi_\alpha(y), \delta))}(x)}{\|\phi_\alpha(Tx) - \phi_\alpha(Ty)\|^{s-k}}\: d\nu(x) d \nu(y) d\alpha \\
		& = \liminf_{\delta \to 0} \frac{C_1}{\delta^k}  \int_{F}  \int_F J_{x,y}\: d\nu(x) d \nu(y),
	\end{align*}
	where
	\[   
	J_{x,y} = \int_{\{ \alpha \in B_m(0,1) : \|\phi_\alpha(x) - \phi_\alpha(y)\| \leq \delta \}}   \frac{d\alpha}{\|\phi_\alpha(Tx) - \phi_\alpha(Ty)\|^{s-k}}, \qquad x,y \in F.
	\]
	To estimate $J_{x,y}$ we proceed as follows. First we note that as $\dim_c T^{k} \nu > 0$, the measure $T^{k} \nu$ has no atoms. Therefore, within the integral $I$ it suffices to consider only points $x,y \in F$ with $T^k x \neq T^k y$. Consequently, estimating $J_{x,y}$ we can assume
	\[   
	x,y \in F, \qquad T^k x \neq T^k y.
	\]
	By Proposition~\ref{prop:decomp}.\ref{it:orbit separation}, we can use Proposition~\ref{prop:singular value on orbit} to obtain
	\begin{equation}\label{eq:sigma k lower bound}
		\sigma_k(D_{x,y}) \geq \frac{C_2(\eps(F))^{2k - 2}}{\Lip(T)} \|T^k x - T^k y\| >0
	\end{equation}
	for some $C_2 > 0$. In particular, this implies that $\rank D_{x,y} = k$ and $\Ker D_{x,y}$ is an $(m - k)$-dimensional subspace of $\R^m$. By \eqref{eq:matrix_form}, we can write 
	\begin{align*}
		\{ \alpha \in B_m(0,1) : \|\phi_\alpha(x) - \phi_\alpha(y)\| \leq \delta \} & = 	\{ \alpha \in B_m(0,1) :  \|D_{x,y}\alpha + w_{x,y}\| \leq \delta \} \\
		& \subset \left\{ \beta + \gamma : \beta \in  B_m(0,1) \cap \Ker D_{x,y},\, \gamma \in G_{x,y}(\delta) \right\},
	\end{align*}
	where
	\[ G_{x,y}(\delta) =  B_m(0,1) \cap (\Ker D_{x,y})^\perp \cap D_{x,y}^{-1}\left( B_k(-w_{x,y}, \delta) \right) \subset \R^m.\]
	Identifying $(\Ker D_{x,y})^\perp$ with $\R^k$, we can use Lemma~\ref{lem: key_ineq_inter} to obtain
	\begin{equation}\label{eq: G measure bound}
		\mH^k(G_{x,y}(\delta)) \leq C_3\left(\frac{\delta}{\sigma_k(D_{x,y})}\right)^k
	\end{equation}
	for some $C_3 > 0$. 
	
	Now we estimate $J_{x,y}$. By \eqref{eq:matrix_form} and Fubini's theorem,
	\begin{equation}\label{eq:preimage integral Fubini}
	J_{x,y} \leq C_4 \int_{G_{x,y}(\delta)} \int_{ \Ker D_{x,y} \cap B_{m}(0,1)} \frac{d\mH^{m-k}(\beta) d\mH^k(\gamma)}{\|D_{Tx,Ty} \beta + D_{Tx, Ty}\gamma + w_{Tx, Ty}\|^{s-k}},
	\end{equation}
where we use the fact that the $m$-dimensional Lebesgue measure satisfies $\Leb = C_4\mH^{m-k} \otimes \mH^k$ for some $C_4 > 0$.
		Fix $\gamma \in G_{x,y}(\delta)$. By Proposition~\ref{prop:decomp}.\ref{it:orbit separation}, we can apply Proposition~\ref{prop:singular value on orbit kernel} to obtain $$\sigma_1(D_{Tx,Ty}|_{\Ker D_{x,y}}) \geq C_5\|T^k x - T^k y\| > 0$$ for some $C_5 > 0$. Consequently,  Lemma~\ref{lem:energy_int_ineq} applied for $p=1$ and $b = D_{Tx, Ty}\gamma + w_{Tx, Ty}$ gives
	\[ \int_{ \Ker D_{x,y} \cap B_{m}(0,1)}  \frac{d\mH^{m-k}(\beta)}{\|D_{Tx,Ty} \beta + D_{Tx, Ty}\gamma + w_{Tx, Ty}\|^{s-k}} \leq \frac{C_6}{\|T^k x - T^k y\|^{s-k}} \]
	for some $C_6 > 0$ (we use here $0 < s - k < 1$). Combining \eqref{eq:sigma k lower bound}, \eqref{eq: G measure bound} and \eqref{eq:preimage integral Fubini} yields
	\[
		J_{x,y}  \leq C_7 \frac{\delta^k}{(\sigma_k(D_{x,y}))^k\|T^k x - T^k y\|^{s-k}} \leq C_8 \frac{\delta^k}{\|T^k x - T^k y\|^s}
	\]
	for some $C_7, C_8 > 0$. Finally, this gives
	\[ I \leq C_1 C_8 \int_F \int_F \frac{d\nu(x) d\nu(y)}{ \|T^k x - T^k y\|^s} = C_1 C_8 \mE_{s}(\nu) < \infty, \]
	as $\cdim \nu \geq \lhdim T^k \mu - \eta > s$ by Proposition~\ref{prop:decomp}.\ref{it:cdim>hdim}. This establishes \eqref{eq:predict slices main integral} and finishes the proof of the theorem.
\end{proof}

\begin{proof}[Proof of Theorem~\rm\ref{thm:predictability sections}]
Assertion~\ref{it:predictability sections lhdim} is a direct consequence of Theorem~\ref{thm:predictability sections lower full}, while assertion~\ref{it:predictability sections uhdim} can be obtained in the same way as Theorem~\ref{thm:injectivity sections}.\ref{it:injectivity sections uhdim} was obtained from Theorem~\ref{thm:injectivity sections}.\ref{it:injectivity sections lhdim} within the proof of Theorem~\ref{thm:injectivity sections}.
\end{proof}

\section{A counterexample -- proof of Theorem~\ref{thm:counterexample}} \label{sec:counterexample}

In this section we prove Theorem~\ref{thm:counterexample}. First, recall the definition of the Smale--Williams solenoid attractor (see e.g.~\cite{Robinson99}). Let $Y$ be a solid torus (a smooth manifold with boundary) defined by $$Y=\mathbb S^1\times \D,$$ where $\mathbb S^1 = \R/(2\pi\Z)$ and $\D$ is the unit disc in the complex plane $\mathbb C$. Let 
$\tilde T \colon Y \rightarrow Y$ be given by 

$$\tilde T(t,z) = \left(2t \mod 2\pi, \tfrac{1}{4}z + \tfrac{1}{2}e^{it}\right).$$
The map $\tilde T$ provides a  $C^\infty$-embedding of $Y$ into itself. Let
\[
\Lambda = \bigcap_{n=0}^\infty \tilde T^n (Y).
\]
The set $\Lambda$ is called the \emph{Smale--Williams solenoid}. It is an invariant compact hyperbolic attractor admitting a natural (SRB) measure $\nu$, in the sense of Definition \ref{defn:nat_measure}. It is known (see \cite{simon1997hausdorff, RamsSimon03}) that
\[
\idim(\nu) = \hdim \nu = \hdim \Lambda = \bdim \Lambda = \frac 3 2.
\]

\begin{thm}\label{thm:axiom_A_extension}
The dynamical system $(Y,\tilde T^2)$ can be $C^\infty$-dynamically embedded into $(M, T)$, where $M$ is a $3$-dimensional manifold and $T\colon M \to M$ is a $C^\infty$-axiom A diffeomorphism.    
\end{thm}
\begin{proof}
	Let $M=L(3,7)$, the \emph{$(3,7)$-lens space} (this smooth $3$-manifold is defined in e.g.~\cite[Example~2.43]{hatcher2002algebraic}). According to \cite[Claim on p.~4377]{jiang2004}, $M$ admits a $C^\infty$-diffeomorphism $T\colon M \to M$ such that $\Omega (T)=A\cup R$, where $A$ is a $4$-adic solenoid attractor ($C^\infty$-conjugated to $(\Lambda,\tilde T^2)$) and $R$ is a  $4$-adic solenoid repeller.\footnote{Note that the definition of an attractor presented in \cite[p.~4373]{jiang2004} is compatible with Definition~\ref{defn:nat_measure}.} By \cite[p.~4375]{jiang2004}, the map $T$ is an axiom~A diffeomorphism see e.g.~\cite[Chapter 3]{BR08} for the definition).
\end{proof}

The proof of Theorem~\rm\ref{thm:counterexample} is split into three propositions.
Consider the manifold $M$ and the diffeomorphism $T\colon M \to M$ from Theorem~\ref{thm:axiom_A_extension}. Denote by $\psi$ a $C^\infty$-embedding of $(Y,\tilde T^2)$ into $(M, T)$. Let 
\[
X = \psi(\Lambda), \qquad \mu = \psi\nu.
\]
Then $X$ is a $T$-invariant compact hyperbolic attractor with natural measure $\mu$ and
\[
\idim(\mu) = \hdim \mu = \hdim X = \bdim X = \frac 3 2.
\]
For the system $(X,T)$ we consider $2$-delay coordinate maps $\phi_h\colon X \to \R^2$ corresponding to Lipschitz observables $h\colon X \to \R$. 
For simplicity, in the two following propositions we use the notation which identifies $(X, T)$ with $(\Lambda,\tilde T^2)$.

\begin{prop}\label{prop:h_0 predict}
	The observable $h_0\colon X \to \R$ given by $h_0(t,z)=\cos t$ is $k$-pre\-dict\-able for every $k \in \N$.
\end{prop}

\begin{proof}
It suffices to prove that $h_0$ is $1$-predictable, i.e. $h_0(t,z)$ determines $h_0(T(t,z))$ uniquely. This follows from the trigonometric identity $\cos 2t = 2\cos^2 t - 1$, which shows that $h_0(t,z) = \cos t$ determines uniquely $\cos 2t$ and hence also $\cos 4t = h_0(T(t,z))$ (or more explicitly: $h_0(T(t,z)) = 2(2h_0(t,z)^2 - 1)^2 - 1$).
\end{proof}

\begin{prop}\label{prop:h_0 non injective}
For every observable $h \in \Lip(X, \R)$, the corresponding $k$-delay coordinate map $\phi_h$ is not injective on $X$ for $k=1,2$.
\end{prop}

\begin{proof}
This follows directly from the fact that the solenoid $\Lambda$ does not embed topologically into $\R^2$, see \cite{Bing60,HoehnOversteegen}.
\end{proof}

\begin{prop}\label{prop:h_0 U}
For every $k \in \N$ and $\delta > 0$ there exists $\eps_0> 0$ and an open set $\mathcal{U} \subset \Lip(X, \R)$ containing $h_0$, such that for every $\eps < \eps_0$ and every $h \in \mathcal{U}$,
		$$\{ (t,z) \in X : \sigma_{h,\eps}(\phi_h(t,z)) > \delta \} = \emptyset,$$
		where $\phi_h$ is the $2$-delay coordinate map corresponding to $h$.
\end{prop}

\begin{proof}
This is an adaptation of the proof of  \cite[Theorem~7.2]{BGSPredict}. By Proposition~\ref{prop:h_0 predict} and \cite[Proposition~1.9]{BGSPredict}, there exists a continuous prediction map $S_{h_0} \colon \phi_{h_0}(X) \to \phi_{h_0}(X)$ satisfying $\phi_{h_0}(T(t,z)) = S_{h_0}(\phi_{h_0}(t,z))$ for every $(t,z) \in X$. Hence, for $\delta > 0$ we can find $0 < \eps_0 < \delta/6$ such that 
\[
\text{if}\quad \|\phi_{h_0}(t,z) - \phi_{h_0}(s,u)\| < 3 \eps_0, \quad \text{then} \quad \|\phi_{h_0}(T(t,z)) - \phi_{h_0}(T(s,u))\| < \frac{\delta}{6} 
\]
for $(t,z), (s,u) \in X$. Let $\eps_1 > 0$ be such that $\| \phi_h - \phi_{h_0} \| < \eps_0$ if $\|h - h_0 \|_{\Lip(X)} < \eps_1$ and let $$\mathcal{U} = \{h \in \Lip(X, \R) : \|h - h_0 \|_{\Lip(X)} < \eps_1\}.$$ With this choice, if $h \in \mathcal{U}$, then for every $\eps < \eps_0$ and $(s,u) \in \phi_{h}^{-1}(B(\phi_h(t,z), \eps))$, we have $\| \phi_{h_0}(t,z) - \phi_{h_0}(s,u) \| \leq \eps + 2 \eps_0 < 3\eps_0$, so $\|\phi_{h_0}(T(t,z)) - \phi_{h_0}(T(s,u))\| \leq \delta / 6$ and hence 
\[
\|\phi_{h}(T(t,z)) - \phi_{h}(T(s,u))\|  < \frac{\delta}{6} +2 \eps_0 < \frac{\delta}{2}.
\]
This together with the definitions of $\chi_{h,\eps}$ and $\sigma_{h,\eps}$ gives $\|\chi_{h,\eps}(\phi_{h}(t,z)) - \phi_{h}(T(t,z))\| \leq \delta/2$ and, consequently, $\sigma_{h,\eps}(\phi_{h}(t,z)) \le \delta$ for every $(t,z) \in X$.

\end{proof}

Propositions~\ref{prop:h_0 non injective}--\ref{prop:h_0 U} complete the proof of  Theorem~\rm\ref{thm:counterexample}.

\section{Examples and discussion on assumptions}\label{sec:examples}

In this section we provide examples showing the necessity of several assumptions within our results.

\subsection{Theorem~\ref{thm:ssoy2.i hdim} does not hold for the information dimension}\label{sec:no id}

In the following proposition we present an example showing that in Theorem~\ref{thm:ssoy2.i hdim}, one cannot replace the assumption $k < \hdim T^k ( \mu |_{ X \setminus \Prep_k (T)})$ by $k < \lid (T^k ( \mu |_{ X \setminus \Prep_k (T)}))$.

\begin{prop}\label{prop:ex ID}
There exists a compact set $X \subset \R^2$, a Borel probability measure $\mu$ on $X$ and a Lipschitz map $T \colon X \to X$ such that the following hold.
\begin{enumerate}[$($a$)$]
	\item $\mu \big( \bigcup_{p = 1}^\infty \Prep_p(T)\big) = 0$,
	\item $\hdim\mu < 1$,
	\item $\lid(\mu) = \lid(T \mu) > 1$.
	\item A prevalent Lipschitz observable $h \colon X \to \R$ is almost surely $1$-predictable, in particular $\lim_{\eps \to 0} \mu(\{x \in X: \sigma_{h,\eps}(\phi_h(x)) > \delta\}) = 0$ for every $\delta > 0$, where $\phi_h$ is the $1$-delay coordinate map corresponding to $h$.
\end{enumerate}

\end{prop}

\begin{proof} It will be convenient for us to consider certain self-similar sets and measures constructed via iterated function systems. For an introduction to this theory see \cite{BSSBook}. Consider an iterated function system $\{f_i\}_{i=0}^3$, where $f_i \colon \R^2 \to \R^2$ are given as 
\[
f_i(x) = \lam x + t_i, \qquad t_0 = (0,0), \; t_1 = (0,1 - \lam), \; t_2 = (1 - \lam, 0), \; t_3 = (1 - \lam,1 - \lam)
\]
for some $1/4 < \lam < 1/2$. Let $X \subset \R^2$, called 
a \emph{four-corner Cantor set}, be the attractor of the system $\{f_i\}_{i=0}^3$, i.e.~the unique non-empty compact set satisfying $X = \bigcup_{i=0}^3 f_i (X)$. 
We have $X = \pi(\Om)$, where $\pi \colon \Om \to \R^2$ is the \emph{natural projection map} from the symbolic space $\Om = \{0,1,2,3\}^\N$, given by $\pi (\om_1, \om_2, \ldots) = \lim_{n \to \infty} f_{\om_1} \circ \cdots \circ f_{\om_n} (0)$ (see e.g.~\cite[Chapter 2]{BP17}). As $\lam < 1/2$, the sets $f_i([0,1]^2)\subset [0,1]^2$ are pairwise disjoint, hence $\pi$ is a bijection. Moreover, it is straightforward to check that if we endow $\Om$ with a metric $d(\om, \tau) = \lam^{|\om \wedge \tau|}$ (where $\om \wedge \tau$ is the longest common prefix of infinite words $\om \neq \tau$ and $|\om \wedge \tau|$ is its length), then $\pi$ is a bi-Lipschitz map and $$\hdim X = \bdim X = \hdim \Om = \bdim \Om = \frac{\log 4}{- \log \lam}$$ (see e.g.~\cite[Theorem~2.2.2]{BP17}). Hence, $1 < \hdim \Om = \bdim \Om < 2$, and following the construction in \cite[Section 3]{FanLauRao02}, one can construct a non-atomic Borel probability measure $\nu$ on $\Om$ with $\supp\nu = \Om$ and $\hdim \nu < 1 < \lid(\nu) = \lid(\sigma \nu)$, where $\sigma \colon \Om^\N \to \Om^\N$ is the left-side shift (in this construction, $\nu$ is a convex combination of infinite product measures).\footnote{ The construction in \cite{FanLauRao02} corresponds to $\Om = \{0,1\}^\N$ with $\lam = 1/2$, but extends directly to our case.} Setting $T:X \to X$ and $\mu$ on $X$ as
$$T = \pi \circ \sigma \circ \pi^{-1}, \qquad \mu = \pi \nu,$$
and using the fact that $\pi$ is bi-Lipschitz and the set $\bigcup_{p = 1}^\infty \Prep_p(\sigma)$ is countable, one obtains
\[
\mu \Big( \bigcup_{p = 1}^\infty \Prep_p(T)\Big) = 0, \qquad \hdim\mu < 1, \qquad \lid(\mu) = \lid(T \mu) > 1,
\]
so $\mu$ satisfies assertions (a)--(c). Note that for $k = 1$, the measure $\mu$ satisfies the assumptions of Theorem~\ref{thm:ssoy2.i hdim}, with condition $\lid(T \mu) > 1$ instead of $\hdim(T \mu) > 1$. On the other hand, as $\hdim\mu < 1$, it follows from \cite[Theorem~1.18]{BGS22} that a prevalent Lipschitz observable $h \colon X \to \R$ is almost surely $1$-predictable, which implies $\lim_{\eps \to 0}\ \mu(\{x \in X: \sigma_{h,\eps}(\phi(x)) > \delta\}) = 0$ for every $\delta > 0$ (see~Remark~\ref{rem:<>dim}). This shows (d).
\end{proof}

\begin{rem}
Note that the measure $\mu$ constructed in Proposition~\ref{prop:ex ID} is not a natural measure for a smooth diffeomorphism (it is even non-invariant), hence it does not provide a counterexample to assertion~\ref{it: ssoy2 i} of the SSOY prediction error conjecture. Finding such a counterexample (or proving that it does not exist) remains an open question.
\end{rem}

\subsection{In Theorems~\ref{thm:ssoy2.i hdim} and~\ref{thm:injectivity sections}--\ref{thm:absolute continuity}, \boldmath $T^k \mu$ (or $T^{k-1} \mu$) cannot be replaced by $\mu$}\label{sec:iterate examples}

We present an example showing that in Theorem~\ref{thm:ssoy2.i hdim}, the assumption $k < \hdim T^k ( \mu |_{ X \setminus \Prep_k (T)} )$ cannot be replaced by $k < \hdim ( \mu |_{ X \setminus \Prep_k (T)})$, and in Theorem~\ref{thm:absolute continuity}, the assumption $k < \lhdim T^{k-1} \mu$ cannot be replaced by $k < \lhdim \mu$. Additionally, the example shows that 
in Theorems~\ref{thm:injectivity sections} and \ref{thm:predictability sections}, the lower bounds in assertions \ref{it:predictability sections lhdim} and \ref{it:predictability sections uhdim} cannot be replaced by $\lhdim \mu - k$ and $\hdim \mu - k - \eps$, respectively. Note that the example also shows that in each of the above cases, one cannot replace $T^k \mu$ (or $T^{k-1} \mu$) by $\mu$ even under the additional assumption of the injectivity of $T$.

\begin{prop}\label{prop:ex Tk mu} There exists a compact set $X \subset \R^3$, a Borel probability measure $\mu$ on $X$ and a Lipschitz injective map $T \colon X \to X$, with $\hdim \mu > 2$ and $\hdim T^k\mu < 1$ for every $k \geq 1$, such that for a prevalent Lipschitz observable $h\colon X \to \R$ and the $2$-delay coordinate map $\phi_h$ corresponding to $h$, the following hold.
\begin{enumerate}[$($a$)$]
	\item $h$ is almost surely $2$-predictable, in particular $\lim_{\eps \to 0} \mu(\{x \in X: \sigma_{h,\eps}(\phi(x)) > \delta\}) = 0$ for every $\delta > 0$.
	\item $\phi_h \circ T(\mu_{h, \phi_h(x)}) = \delta_{\phi_h(Tx)}$, and hence $\hdim \phi_h \circ T(\mu_{h, \phi_h(x)}) = 0$ for $\mu$-almost every $x \in X$.
	\item The measure $\phi_h \mu$ is supported on a set of Hausdorff dimension smaller than $2$, so it is not absolutely continuous with respect to the $2$-dimensional Lebesgue measure in $\R^2$.
\end{enumerate}
\end{prop}

\begin{proof}
We define $X \subset \R^3$ to be a compact set, which is a disjoint union $X = \bigcup_{j = 1}^\infty X_j \cup \{ 0 \}$, where
\begin{enumerate}[(i)]
	\item $X_1$ is a compact self-similar set with $\hdim X_1 = \bdim X_1 > 2$,
	\item $X_2$ is a compact self-similar set with $\hdim X_2 = \bdim X_2 < 1$,
	\item $X_1$ and $X_2$ are homeomorphic by a Lipschitz bijection $T \colon X_1 \to X_2$,
	\item $X_j = f_j(X_2)$ for $j \geq 3$, where $f_n \colon \R^3 \to \R^3$ is a similarity with scaling ratio $2^{-j}$ (in particular $X_n$ and $X_2$ are bi-Lipschitz homeomorphic),
	\item $\lim_{n \to \infty} \dist(X_n, \{ 0\}) = 0$.
\end{enumerate}
The sets $X_1$ and $X_2$ can be constructed similarly as in Section~\ref{sec:no id}, both being homeomorphic to the same symbolic space $\Omega$, but with different contraction rates. Using (i)--(v), we can define a Lipschitz injective map $T \colon X \to X$ such that $T(X_j) = X_{j+1}$ for $j \geq 1$ and $T(0)=0$. Let $\mu$ be a probability measure on $X_1$ satisfying $\hdim \mu = \lhdim \mu = \hdim X_1 > 2$ (one can take $\mu$ to be the dimension maximizing self-similar measure on $X_1$, see e.g.~\cite[proof of Theorem~2.2.2]{BP17} or \cite[Theorem~5.2.5]{edgar}). Note that $\hdim T^j\mu \leq \hdim X_2 < 1$ for every $j \geq 1$. As $\hdim T \mu < 1$, the map $T$ is injective and $0$ is the only periodic point of $T$, we can apply \cite[Theorem~1.2]{BGS20} to conclude that for a prevalent Lipschitz observable $h \colon X \to \R$, there exists a set $X_h \subset X_2$ of full $T\mu$-measure, on which $h$ is injective. As $T$ is injective, this implies that the $2$-delay coordinate map $\phi_h \colon X \to \R^2$ is injective on the set $T^{-1}(X_h)$ of full $\mu$-measure. Indeed, if $\phi_h(x) = \phi_h(y)$ for $x,y \in T^{-1}(X_h)$, then $h(Tx) = h(Ty)$, hence $Tx = Ty$ and, consequently, $x = y$. Therefore, $\mu_{h,\phi_h(x)} = \delta_x$ for $\mu$-almost every $x \in X$. This immediately implies assertions (a)--(b). To show~(c), note that $\phi_h \mu$ is supported on $\R \times h(T (X_1)) = \R \times h(X_2)$ and $\hdim( \R \times h(X_2)) \leq 1 + \udim X_2 < 2$.
\end{proof}

\subsection{Assumptions on pre-periodic points are necessary}

As noted in Remark~\ref{rem:assum preper}, to see that some assumptions on the size of the set of preperiodic points of $T$ are necessary, it is enough to consider the case when the map $T$ is the identity. In this case every observable is predictable, regardless of the dimension of the phase space and the measure $\mu$. Moreover, the measure $\phi_h \mu$ is supported on a $1$-dimensional diagonal in $\R^k$, hence it cannot be absolutely continuous with respect to the $k$-dimensional Lebesgue measure for $k > 1$. One can easily construct similar examples, where all the points of the phase space are (pre-)periodic with the same period $p$, $p > 1$. Therefore, one cannot remove the assumption on pre-periodic points in Theorems~\ref{thm:ssoy2.i hdim}, \ref{thm:predictability sections}--\ref{thm:absolute continuity} and~Theorem~\ref{thm:loc dim} below.

\appendix

\section{Local dimension projection theorem for delay coordinate maps}\label{app:loc dim}

We prove the following result on the local dimensions of the push-forward of Borel probability measures on compact sets in $\R^N$ by delay coordinate maps. The extends \cite[Theorem~3.5, Remark 4.4]{SauerYorke97} to a more general case.
	
	\begin{thm}\label{thm:loc dim}
		Let $X \subset \R^N$, $N \in \N$, be a compact set, let $\mu$ be a Borel probability measure on $X$ and let $T\colon X \to X$ be a Lipschitz map. Fix $k \in \N$ and $($in the case $k>1)$ assume $\mu(\Prep_{k-1}(T)) = 0$. Then for a prevalent Lipschitz observable $h \colon X \to \R$ and the $k$-delay coordinate map $\phi_h$ corresponding to $h$,
		\[ \ld(\phi_h \mu, \phi_h(x)) \geq \min\{k, \ld(T^{k-1} \mu, T^{k-1} x) \} \quad \text{for } \mu\text{-almost every } x \in X. \]
		If, additionally, $\mu$ is $T$-invariant or $T$ is bi-Lipschitz onto its image, then for a prevalent Lipschitz observable $h \colon X \to \R$,
		\[ \ld(\phi_h \mu, \phi_h(x)) = \min\{k, \ld(\mu, x) \}  \quad \text{for } \mu\text{-almost every } x \in X. \]
	\end{thm}
	
	Theorem~\ref{thm:loc dim} provides the following corollary.
	
	\begin{cor}\label{cor:loc dim}
	Let $X \subset \R^N$, $N \in \N$, be a compact set, let $\mu$ be a Borel probability measure on $X$ and let $T\colon X \to X$ be a Lipschitz map, such that $\mu$ is $T$-invariant or $T$ is bi-Lipschitz onto its image. Fix $k \in \N$ and $($in the case $k>1)$ assume $\mu\left( \bigcup \limits_{p=1}^{k-1}\Per_{p}(T) \right) = 0$. Suppose that the local dimension $d(\mu, x)$ exists for $\mu$-almost every $x \in X$. Then for a prevalent Lipschitz observable $h \colon X \to \R$, the local dimension $d(\phi_h \mu, \phi_h(x))$ exists at $\mu$-almost every $x \in X$ and satisfies
		\[ d(\phi_h \mu, \phi_h(x)) =  \min\{k, d(\mu, x) \}, \]
		where $\phi_h$ is the $k$-delay coordinate map corresponding to $h$.
	\end{cor}

		To show Theorem~\ref{thm:loc dim}, we first prove the following more specific result.
		
		\begin{thm}\label{thm:loc dim full}
		Let $X \subset \R^N$, $N \in \N$, be a compact set, let $\mu$ be a Borel probability measure on $X$ and let $T\colon X \to X$ be a Lipschitz map. Fix $k \in \N$ and $($in the case $k>1)$ assume $\mu(\Prep_{k-1}(T)) = 0$.
		Fix $d \geq 2k-1$ and let $\{h_1, \ldots, h_m\}$ be the set of monomials of $N$ variables of degree at most $d$. Let $h \colon X \to \R$ be a Lipschitz observable. Then for the $k$-delay coordinate map $\phi_h$ corresponding to $h$ and Lebesgue-almost every $\alpha \in \R^m$,
			\[ \ld(\phi_\alpha \mu, \phi_\alpha(x)) \geq \min\{k, \ld(T^{k-1} \mu, T^{k-1} x) \}  \quad \text{for } \mu\text{-almost every } x \in X. \]
		\end{thm}

		\begin{proof} 
		The proof follows the ideas used in \cite[Theorem~3.5]{SauerYorke97} and \cite[Theorem~4.1]{HuntKaloshin97}. As previously, it is sufficient to prove the assertion for Lebesgue-almost every $\alpha \in B_m(0,1)$. 
			Fix $\eta > 0$ and consider the decomposition $\mF$ from Proposition~\ref{prop:decomp} for $\ell = k -1$ (actually, in this proof we only make use of the properties~\ref{it:F sum} and~\ref{it:orbit separation} of the decomposition, hence $\eta$ will not be used in the proof). Fix $0 < s < k$. For $M>0$ define
			\[  F_M = \{ x \in F : \mE_s(T^{k-1} \mu, T^{k-1} x) \leq M \}. \]
			We will show that 
			\[         
			I  = \int_{B_m(0,1)} \int_{F_M} \mE_s(\phi_\alpha(\mu|_F), \phi_\alpha(x))d\mu(x)d\alpha < \infty.
			\]
To do it, note that by \eqref{eq:matrix_form}, Tonelli's theorem and Lemma~\ref{lem:energy_int_ineq},
			\[
			I  = \int_{B_m(0,1)} \int_{F_M} \int_{F} \frac{d\mu(y)d\mu(x)d\alpha}{\| \phi_\alpha(x) - \phi_\alpha(y) \|^s} 
			 =\int_{B_m(0,1)} \int_{F_M} \int_{F} \frac{d\mu(y)d\mu(x)d\alpha}{\| D_{x,y}\alpha + w_{x,y} \|^s}
			\leq C_1 \int_{F_M} \int_{F} \frac{d\mu(y)d\mu(x)}{(\sigma_k(D_{x,y}))^s}
			\]
			for some $C_1 > 0$. By Proposition~\ref{prop:decomp}\ref{it:orbit separation}, we can use Proposition~\ref{prop:singular value on orbit} to obtain
			\[
			\begin{split} I &\leq C_2 \int_{F_M} \int_{F} \frac{d\mu(y)d\mu(x)}{\|T^{k-1} x - T^{k-1}y \|^s} \leq C_2 \int_{F_M} \int_{X}  \frac{d\mu(y)d\mu(x)}{\|T^{k-1} x - T^{k-1}y \|^s}\\
				& = C_2 \int_{F_M} \Ek_s(T^{k-1} \mu, T^{k-1} x) d \mu(x) \leq C_2M < \infty.
			\end{split}
			\]
			Letting $M \to \infty$ and applying Lemma~\ref{lem:loc dim energy}, we obtain, for Lebesgue-almost every $\alpha \in B_m(0,1)$,
			\[ \ld(\phi_\alpha(\mu|_F), \phi_\alpha(x)) \geq s \quad \text{for } \mu\text{-almost every } x \in F \text{ such that } \mE_s(T^{k-1} \mu, T^{k-1} x) < \infty. \]
			Applying Lemma~\ref{lem:loc dim energy} once more, we arrive at
			\[ \ld(\phi_\alpha(\mu|_F), \phi_\alpha(x)) \geq s \quad \text{for } \mu\text{-almost every } x \in F \text{ such that } \ld(T^{k-1} \mu, T^{k-1} x) > s. \]
			Taking intersection over all rational $s \in (0,k)$ we obtain, for Lebesgue-almost every $\alpha \in B_m(0,1)$,
			\[ \ld(\phi_\alpha(\mu|_F), \phi_\alpha(x)) \geq \min \{ k,  \ld(T^{k-1} \mu, T^{k-1} x) \} \quad \text{for } \mu\text{-almost every } x \in F. \]
			As  $\mu \left(\bigcup_{F \in \mF} F\right) = 1$, to end the proof it is sufficient to show  
			\begin{equation}\label{eq:loc dim restriction}
				\ld(\phi_\alpha(\mu|_F), \phi(x)) = \ld(\phi_\alpha \mu, \phi(x)) \quad \text{for } \mu\text{-almost every } x \in F.
			\end{equation}
			To prove  \eqref{eq:loc dim restriction}, note that $\phi_\alpha(\mu|_F) \ll \phi_\alpha \mu$, so by the differentiation theorem for measures (see e.g.~\cite[Theorem~2.12]{mattila}), the Radon-Nikodym derivative 
			\[\frac{d\phi_\alpha(\mu|_F)}{d\phi_\alpha \mu}(y) = \lim_{\delta \to 0} \frac{\phi_\alpha(\mu|_F)(B(y,\delta))}{\phi_\alpha \mu(B(y,\delta))}  \]
			exists and is positive and finite for $\phi_\alpha(\mu|_F)$-almost every $y \in \R^k$.
			Therefore, for $\phi_\alpha(\mu|_F)$-almost every $y \in \R^k$,
			\[ \ld(\phi_\alpha \mu, y ) = \liminf_{\delta \to 0} \frac{\log \left(\phi_\alpha(\mu|_F)(B(y,\delta))\right) + \log \frac{\phi_\alpha \mu(B(y,\delta))}{ \phi_\alpha(\mu|_F)(B(y,\delta))} }{\log \delta} = \ld(\phi_\alpha(\mu|_F), y). \]
			In particular, this holds with $y = \phi(x)$ for $\mu$-almost every $x \in F$, which proves \eqref{eq:loc dim restriction}.
		\end{proof}

		\begin{proof}[{Proof of Theorem~\rm\ref{thm:loc dim}}]
			The main assertion follows directly from Theorem~\ref{thm:loc dim full}. The additional ones use the observation that if $\mu$ is $T$-invariant or $T$ is bi-Lipschitz onto its image, then
			\begin{equation}\label{eq:dim invariance} \ld(T^{k-1} \mu, T^{k-1}x) = \ld(\mu, x) \quad  \text{for } \mu\text{-almost every } x \in X.
			\end{equation}
			Indeed, if $\mu$ is $T$-invariant, then $T^{k-1} \mu = \mu$, and as $T$ is Lipschitz, we have $B(x, r/\Lip(T)) \subset T^{-1}(B(Tx,r))$, so $\mu(B(x, r/\Lip(T)) \leq \mu(B(Tx, r))$ by the invariance of $\mu$. Therefore, $\ld(\mu, Tx) \leq \ld(\mu, x)$.
			On the other hand, by the invariance of $\mu$, there holds $\int \ld(\mu, Tx)d\mu(x) = \int \ld(\mu, x)d\mu(x)$. This proves \eqref{eq:dim invariance}. If $T$ is bi-Lipschitz, then $B(x, r/\Lip(T)) \subset T^{-1}(B(Tx,r)) \subset B(x, \Lip(T^{-1})r)$, so $\ld(T \mu, Tx) = \ld(\mu, x)$ at every $x \in X$ and \eqref{eq:dim invariance} follows. Moreover, as $\phi_\alpha$ is Lipschitz for every Lipschitz observable $h$, the same arguments as above used for $\phi_\alpha$ instead of $T^{k-1}$ provide
			\[ \ld(\phi_\alpha \mu, \phi_\alpha(x)) \leq \min\{k, \ld(\mu, x)\}  \quad \text{for } \mu\text{-almost every } x \in X. \]
			Combining this with \eqref{eq:dim invariance} and Theorem~\ref{thm:loc dim full} ends the proof of the additional assertions.
		\end{proof}
		
		\begin{proof}[{Proof of Corollary~\rm\ref{cor:loc dim}}]
			In the same way as in the proof of Theorem~\ref{thm:loc dim}, for a prevalent Lipschitz observable $h \colon X \to \R$ we obtain the upper bound
			\[ \ud(\phi_h \mu, \phi(x)) \leq \min\{k, \ud(\mu, x)\} = \min\{k, d(\mu, x)\}  \quad \text{for } \mu\text{-almost every } x \in X. \]
			By Theorem~\ref{thm:loc dim}, we also have
			\[ \ld(\phi_h \mu, \phi(x)) = \min\{k, \ld(\mu, x)\} = \min\{k, d(\mu, x)\} \quad \text{for } \mu\text{-almost every } x \in X. \]
			Combining these facts finishes the proof.
		\end{proof}

\bibliographystyle{alpha}
\bibliography{universal_bib}

\end{document}